\documentclass[english]{amsart}
\usepackage{url}
\usepackage{amsopn}                    
\usepackage{amsmath}
\usepackage{amssymb}
\usepackage{stmaryrd}
\usepackage[all,2cell,emtex]{xy}
\usepackage[mathscr]{euscript}
\usepackage{mathrsfs}
\usepackage{bbm}
\DeclareFontFamily{OT1}{pzc}{}
\DeclareFontShape{OT1}{pzc}{m}{it}%
             {<-> s * [1,150] pzcmi7t}{}
\DeclareMathAlphabet{\mathpzc}{OT1}{pzc}%
                                 {m}{it}

\theoremstyle{plain}
\newtheorem*{thmi}{Theorem \ref{cor:homotcohnervequillenequiv}}
\newtheorem*{propi}{Proposition \ref{prop:embeddingclassical}}
\newtheorem*{cori}{Corollary \ref{cor:homotcohnervequillenequivsimpl}}
\newtheorem{thm}{Theorem}[section]
\newtheorem{prop}[thm]{Proposition}
\newtheorem{lemma}[thm]{Lemma}
\newtheorem{cor}[thm]{Corollary}

\theoremstyle{remark}
\newtheorem{rem}[thm]{Remark}

\theoremstyle{definition}
\newtheorem{definition}[thm]{Definition}
\newtheorem{paragr}[thm]{}

\theoremstyle{plain} 
\swapnumbers

\numberwithin{equation}{thm}

\renewcommand{\mathcal}{\mathpzc}
\renewcommand{\mathbb}{\mathbbm}

\renewcommand{\leq}{\leqslant}
\renewcommand{\geq}{\geqslant}

\newcommand{\To}{\longrightarrow}

\newcommand{\cat}{\mathpzc{Cat}}

\newcommand{\Hom}{\operatorname{\mathrm{Hom}}}
\newcommand{\sHom}{\operatorname{\mathpzc{Hom}}}

\newcommand{\nerf}{N}

\newcommand{\op}[1]{{#1}^{\mathit{op}}}

\newcommand{\ex}{\operatorname{\mathit{Ex}}}

\newcommand{\ho}{\operatorname{\mathbf{Ho}}}
\newcommand{\ob}{\operatorname{\mathrm{Ob}}}

\newcommand{\smp}[1]{ \Delta[#1]}

\newcommand{\derR}{\mathbf{R}}

\newcommand{\bord}{\partial}

\def\TO#1{\mathrel{\hbox to #1pt{\rightarrocharReedyfibrantsegaloperwfill}}}
\def\OT#1{\mathrel{\hbox to #1pt{\leftarrowfill}}}

\def\limind{\mathop{\oalign{\rm lim\cr
\hidewidth$\longrightarrow$\hidewidth\cr}}}%
\renewcommand{\varinjlim}{\limind}%

\renewcommand{\to}{\To}
\newcommand{\todouble}{\xymatrixcolsep{1pc}\xymatrix{\ar@<.5ex>[r]\ar@<-.5ex>[r]&}}
\newcommand{\todoubleop}{\xymatrixcolsep{1pc}\xymatrix{\ar@<.5ex>[r]&\ar@<.5ex>[l]}}

\renewcommand{\hookrightarrow}{{\hskip -1.5pt\raise 1.5pt\vbox{\xymatrixcolsep{.9pc}\xymatrix{\ar@{^{(}->}[r]&}}\hskip -3.5pt}}

\newcommand{\intcoin}[5]{\raise 12pt\vbox{\xymatrixcolsep{.9pc}\xymatrixrowsep{.7pc}\xymatrix{%
\scriptstyle #1\ar[r]^{\scriptscriptstyle #5}\ar[d]_{\scriptscriptstyle #4}&\scriptstyle #3\\\scriptstyle #2}}}

\newcommand{\set}{\mathpzc{Set}}
\newcommand{\sset}{\mathpzc{sSet}}
\newcommand{\dset}{\mathpzc{dSet}}
\newcommand{\rdset}{\dset_{\ast}}
\newcommand{\sdset}{\mathpzc{sdSet}}
\newcommand{\trees}{\Omega}
\newcommand{\oper}{\mathpzc{Operad}}
\newcommand{\roper}{\mathpzc{sOpe}\mathpzc{r}_\ast}
\newcommand{\soper}{{\mathpzc{sOper}}}
\newcommand{\scat}{\mathpzc{sCat}}
\newcommand{\preoper}{\mathpzc{PreOper}}

\newcommand{\precat}{\mathpzc{PreCat}}

\newcommand{\Map}{\mathit{Map}}

\newcommand{\spine}{\mathrm{Sc}}

\newcommand{\bv}{W_!}
\newcommand{\bvoper}{W}
\newcommand{\hnerf}{\mathit{hcN}_d}
\newcommand{\ta}{\mathbf{A}}
\newcommand{\tc}{\mathbf{C}}
\newcommand{\ota}{\overline{\ta}}
\newcommand{\otc}{\overline{\tc}}

\newcommand{\specialbullet}{\cdot}

\title{Dendroidal sets and simplicial operads}

\author[D.-C. Cisinski]{Denis-Charles Cisinski}
\address{Universit\'e Paul Sabatier\\
Institut de Math\'ematiques de Toulouse\\
118~route de Narbonne\\
31062~Toulouse cedex~9\\France}
\email{denis-charles.cisinski@math.univ-toulouse.fr}
\urladdr{http://www.math.univ-toulouse.fr/~dcisinsk/}

\author[I. Moerdijk]{Ieke Moerdijk}
\address{Radboud Universiteit Nijmegen\\
Institute for Mathematics, Astrophysics, and Particle Physics\\
Heyendaalseweg~135, 6525~AJ~Nijmegen\\
The~Netherlands}
\email{i.moerdijk@math.ru.nl}
\urladdr{http://www.math.uu.nl/people/moerdijk/}

\begin{document}
\begin{abstract}
We establish a Quillen equivalence relating the homotopy theory of Segal operads and the
homotopy theory of simplicial operads, from which we deduce that the
homotopy coherent nerve functor is a right Quillen equivalence from the model category of simplicial operads
to the model category structure for $\infty$-operads on the category of dendroidal sets.
By slicing over the monoidal unit, this also gives the Quillen equivalence between
Segal categories and simplicial categories proved by J.~Bergner, as well as the
Quillen equivalence between quasi-categories and simplicial categories proved by
A.~Joyal and J.~Lurie. We also explain how this
theory applies to the usual notion of operad (i.e. with a single colour) in the category of spaces.
\end{abstract}
\maketitle

\section*{Introduction}

This paper is the last in a series of three, whose main goal is to establish a
homotopy theoretic equivalence between simplicial operads and dendroidal sets.
When we talk about operads in this paper, we will generally mean coloured operads,
i.e. operads controlling algebraic structures in which there are elements of different ``types'',
such as the structure consisting of a ring together with a module over it, or the structure
of an (enriched) category given by the various hom sets, for which the types are given
by pairs of objects, the domain and the codomain. The more classical uncoloured (``monochromatic'')
operads are included as those coloured operads in which there is only one colour.
The category $\dset$ of dendroidal sets is an extension of the category $\sset$ of simplicial sets.
It is related by a nerve functor to the category $\oper$ of operads in exactly the same way as the category
of simplicial sets is related to the category $\cat$ of small categories.  In fact, there is a special
dendroidal set $\eta$, related to the trivial operad, which we also denote $\eta$,
for which there are canonical isomorphisms of categories
\begin{equation}\tag{I}
\dset/\eta  = \sset
\end{equation}
\begin{equation}\tag{II}
\oper/\eta=\cat\, ,
\end{equation} 
while the corresponding  forgetful functors 
$$i_!: \sset\to\dset\quad\text{and}\quad
j_! : \cat\to\oper$$
are fully faithful embeddings.
 
The dendroidal nerve functor $\nerf_d$ extends the classical nerve functor $\nerf$ in the sense
that the diagram  
$$\xymatrix{
\oper\ar[r]^{\nerf_d}&\dset\\
\cat\ar[u]^{j_!}\ar[r]^\nerf&\sset\ar[u]_{i_!}
}$$
commutes. In \cite{dend3}, we proved that the category of dendroidal sets carries a left proper Quillen
model structure, which under the identification (I) induces
the Joyal model structure on $\sset$. In the Joyal model category structure,
the fibrant objects are the $\infty$-categories,
also known as quasi-categories; see \cite{joyal,lurie}.
The fibrant objects in the model structure on dendroidal sets are similarly referred to as $\infty$-operads.
In fact, there is a bit more structure around: there is a symmetric closed monoidal
structure on $\dset$, reminiscent of the Boardman-Vogt tensor
product of operads \cite{BV}, and compatible with this Quillen structure
in the precise sense that it makes $\dset$ into a monoidal model category.
It turns out (see Theorem \ref{thm:cmcsoper}) that the category $\soper$
of simplicial coloured operads also
carries a Quillen model structure, which, under the simplicially enriched version of
the identification (II) relating simplicial operads and simplicial categories, induces the
Bergner model structure \cite{bergnera} on the category $\scat$ of simplicial categories.
The main result of this paper, already announced in \cite{dend3}, can then be stated as follows:
 
\begin{thmi}
There is a Quillen equivalence
$$\bv : \dset \rightleftarrows \soper: \hnerf$$
between dendroidal sets and simplicial operads. 
\end{thmi}

The right Quillen functor $\hnerf$ is a homotopy coherent version
of the dendroidal nerve functor, and its left adjoint $\bv$ is closely related to the
Boardman-Vogt resolution of operads \cite{BV,BMOp0,BMOp1}.
In fact, the counit of the adjunction 
$$\bv\hnerf (P) \to P$$
is essentially the Boardman-Vogt resolution of $P$. 
On the other hand, for a cofibrant and fibrant dendroidal set $X$, the unit $X\to \hnerf \bv (X)$
can be viewed as a strictification, or rectification, of the $\infty$-operad $X$ by a ``strict'' operad
$\bv(X)$. The force of this theorem is illustrated by the fact that, by considering the corresponding slice
categories over $\eta$,  it immediately implies one of the cornerstone facts in
the theory of $\infty$-categories, proved by Joyal and Lurie \cite{lurie}:
 
\begin{cori}
There is a Quillen equivalence 
$$\bv: \sset \rightleftarrows \scat: \mathit{hcN}$$
between simplicial sets and simplicial categories.
\end{cori}
 
Recall that a reduced simplical set is one with a unique vertex.
There is an analogous notion of reduced dendroidal set, and our  theorem
will also be seen to imply a similar equivalence between classical
(monochromatic) simplicial operads and a related Quillen model
category structure on these reduced dendroidal sets:
 
\begin{propi}
The category of reduced dendroidal sets carries an
induced model structure which is Quillen equivalent to the category of (uncoloured)
simplicial operads (with the model structure of \cite{BMOp0}).
\end{propi}
 
The category of dendroidal sets is a category of presheaves of sets. In our earlier
paper \cite{dend6}, we studied the related category of presheaves of simplicial sets: this is the
category of dendroidal spaces, identical to the category of simplicial objects in $\dset$,
and denoted $\sdset$.  It contains as a full subcategory the category of preoperads --- those
dendroidal spaces whose space of vertices (or objects, or colours)  is discrete. In \cite{dend6},
we proved that the category of dendroidal spaces carries a Rezk style Quillen model structure whose
fibrant objects are referred to as dendroidal complete Segal spaces. We  established a Quillen equivalence
between the original model category of dendroidal sets and this model category of dendroidal
complete Segal spaces.  We also proved that this model category  restricts to a Quillen equivalent
model category structure on the category $\preoper$
of preoperads. Thus, together with the main theorem in
this paper, we obtain a square of Quillen equivalent model categories
\begin{equation}\tag{$\ast$}
\begin{split}
\xymatrix{
\preoper\ar[d]_{\gamma^*}&\soper\ar[l]_{\nerf_d}\ar[d]^{\hnerf}\\
\sdset&\dset\ar[l]_{i}
}\end{split}
\end{equation}
in which $\nerf_d$ is the simplicial enrichment of the dendroidal
nerve functor, while $\gamma^*$ and $i$ are the inclusion functors.
 
The way we prove our theorem is based on these earlier results, and proceeds as follows:
We prove that the functors $\bv$ and $\hnerf$ of the main theorem form a Quillen pair
(Prop. \ref{prop:BVleftQuillen}),
and we prove that the square ($\ast$) is commutative in a homotopy theoretic sense, even though the
top horizontal functor is not a right Quillen functor (although it does preserve weak equivalences).
Thus, using the fact, from our
earlier paper \cite{dend6}, that the inclusion functors
relating $\dset$, $\sdset$ and $\preoper$ are left Quillen equivalences,
we see  that, to prove that $\bv$ and $\hnerf$ form a
Quillen equivalence, it  is in fact enough to prove that the adjunction between preoperads and simplicial
operads is a Quillen equivalence.  In order to do this, we change the model structure on
$\preoper$ used in \cite{dend6}  into a ``tame'' model structure, with the same weak equivalences but
considerably fewer cofibrant objects (Section \ref{sect:7}).  With fewer cofibrant objects to deal with, it is
possible to show that the functor $\nerf_d$ on top of diagram ($\ast$) is a right
Quillen equivalence (Theorem \ref{thm:quillenequivnervepreopersoper}).
This last argument is based on a fundamental property of the nerve functor, stating that it
preserves certain pushouts ``up to an inner anodyne extension'' (in particular, up to a trivial
cofibration); see Prop. \ref{prop:addfreecellweakequiv}.
This fundamental property also implies that $\Sigma$-cofibrant operads
behave like cofibrant operads with respect to homotopy pushouts (Theorem \ref{thm:sigmacofleftproper}).
An immediate consequence (Cor. \ref{cor:propersimpcat})
is that the model category structure on the category of simplicial
categories is in fact left proper, a special feature which doesn't seem to have been observed before.
In the same spirit, we see that any reasonable model of the operad $E_\infty$ defines
a proper model category which is Quillen equivalent to the model category of
simplicial operads (Cor. \ref{cor:quillenequivsoperproper}).
 
The plan of this paper, then, is as follows:  In the first section, we give a detailed proof of the
existence of a closed model structure on the category of simplicial operads. We have known this model
structure for quite a while, and in fact already announced it in our first paper \cite{dend3}.
In the meantime, it has been proved and used independently by M. Robertson \cite{Rob}.
The second section reviews basic definitions and
facts concerning dendroidal sets; the reader can find detailed treatments in the
references \cite{dend1,dend2,dend3}. In section \ref{sect:3}, we prove the fundamental property,
already mentioned, that the dendroidal nerve functor from operads to dendroidal sets preserves
certain pushouts up to homotopy. In Section \ref{sect:4}, we take up the construction of a
generalized Boardman-Vogt resolution of operads from \cite{BMOp0,BMOp1}, and use
it to construct the Quillen pair featuring in our main theorem. Section \ref{sect:5} is essentially
a review of those definitions and facts needed in this paper about dendroidal complete
Segal spaces and about preoperads. In Section \ref{sect:6}, we examine weak equivalences
between preoperads which satisfy some Segal type fibrancy conditions.
In Section \ref{sect:7}, we explain how to modify the model structure on preoperads,
in such a way that the nerve functor from simplicial operads to preoperads
becomes right Quillen, while the homotopy category of preoperads doesn't change.
In Section \ref{sect:8}, we collect all the results together, and deduce our main theorem and
some of its variations and consequences. Finally, Section \ref{sect:9} deals with the reduced
case mentioned above.  

\section{The homotopy theory of simplicial operads}\label{sect:1}

\begin{paragr}\label{def:folkcmcoperads}
Let us denote by $\oper$ the category of coloured operads (=symmetric multicategories)
in the category of sets. The objects of $\oper$ will be simply called \emph{operads}.

A morphism of operads $u:P\to Q$ is said to be \emph{fully faithful} if, for any integer $n\geq 0$, and for
any $(n+1)$-tuple of objects $(x_1,\ldots,x_n,x)$ of $P$, the map
$$P(x_1,\ldots,x_n;x)\to Q(u(x_1),\ldots,u(x_n);u(x))$$
is bijective. A morphism of operads $u:P\to Q$ is \emph{essentially surjective} if, for any
object $y$ in $Q$, there exists an object $x$ in $P$ as well as an isomorphism
$u(x)\simeq y$ in $Q$.
A morphism of operads $u:P\to Q$ is an \emph{isofibration} if, for any isomorphism
$b:y\to y'$ in $Q$ and for any object $x$ in $P$ such that $u(x)=y$, there exists an isomorphism
$a:x\to x'$ in $P$ such that $u(a)=b$ (and $u(x')=y'$).

We recall the following basic result, stated in \cite{dend1} (a full proof can be found in
\cite{weissthesis}).
\end{paragr}

\begin{thm}\label{thm:folkcmcoperads}
The category of operads is endowed with a proper cofibrantly generated
model category structure defined as follows.
The weak equivalences are the morphisms of operads which are fully faithful and essentially
surjective, the cofibrations are the morphisms which are injective on objects, while
the fibrations are the isofibrations.
\end{thm}

\begin{paragr}\label{paragr:mainclassesofmapssimpoper}
Now, consider the category $\soper$ of simplicial operads (i.e. of coloured operads
enriched in simplicial sets).

There is an adjunction
\begin{equation}\label{eq:adjsoperoper}
\pi_0:\soper\rightleftarrows\oper:\iota
\end{equation}
where the right adjoint $\iota$ sends an operad $P$ to itself, while the left adjoint $\pi_0$
sends a simplicial operad $\mathcal P$ to the operad $\pi_0(\mathcal P)$ defined by
$\ob \pi_0(\mathcal P)=\ob \mathcal P$, while, for any integer $n\geq 0$ and
any $(n+1)$-tuple of objects, $\pi_0(\mathcal P)(x_1,\ldots,x_n;x)$ is the set
$\pi_0(\mathcal P(x_1,\ldots,x_n;x))$ of connected components of the simplicial set of operations
$\mathcal P(x_1,\ldots,x_n;x)$. We will have to consider the following kinds of morphisms of simplicial
operads.
\end{paragr}

\begin{definition}\label{def:mainclassesofmapssimpoper}
A morphism of simplicial operads $u: \mathcal P\to\mathcal Q$ is said to be
\emph{fully faithul} if, for any integer $n\geq 0$ and
for any $(n+1)$-tuple of objects $(x_1,\ldots,x_n,x)$ of $\mathcal P$, the map
$$\mathcal P(x_1,\ldots,x_n;x)\to \mathcal Q(u(x_1),\ldots,u(x_n);u(x))$$
is a weak equivalence (in the sense of the usual Quillen model structure on the category of simplicial sets).
 A morphism of simplicial operads $u: \mathcal P\to\mathcal Q$ is said to be
\emph{essentially surjective} if the associated morphism of operads $\pi_0(u):\pi_0(\mathcal P)\to\pi_0(\mathcal Q)$
is essentially surjective.

A morphism of simplicial operads $u: \mathcal P\to\mathcal Q$ is a
\emph{local fibration} (a \emph{local trivial fibration}) if, for any non-negative integer $n$
and for any $(n+1)$-tuple of objects $(x_1,\ldots,x_n,x)$ of $\mathcal P$, the map
$$\mathcal P(x_1,\ldots,x_n;x)\to \mathcal Q(u(x_1),\ldots,u(x_n);u(x))$$
is a Kan fibration (a trivial fibration, respectively).
A morphism of simplicial operads $u: \mathcal P\to\mathcal Q$ is
an \emph{isofibration} if it is a local fibration and if the induced map
of operads $\pi_0(u): \pi_0(\mathcal P)\to\pi_0(\mathcal Q)$ is an isofibration
in the sense of \ref{def:folkcmcoperads}.

A simplicial operad $\mathcal P$ is \emph{fibrant} if the morphism from
$\mathcal P$ to the terminal operad is a local fibration (or, equivalenly,
an isofibration). 
\end{definition}

\begin{paragr}\label{def:cmcfixedcoloursoper}
Let $C$ be a set. We denote by $\soper_C$ the category of simplicial operads with $C$
as a fixed set of objects: the objects of $\soper_C$ are the simplicial operads $\mathcal P$
such that $\ob \mathcal P=C$, and the morphisms of $\soper_C$ are the morphisms
of simplicial operads $\mathcal P \to \mathcal Q$ for which the induced map
$C=\ob \mathcal P\to \ob \mathcal Q=C$ is the identity.
Note that any morphism of $\soper_C$ is in particular essentially surjective.
\end{paragr}

\begin{rem}\label{rem:nicefibrantreplacementsoper}
For any simplicial operad $\mathcal P$, there is a functorial map
\begin{equation}\label{eq:soperfibrantreplace}
\mathcal P\to \ex^\infty(\mathcal P)
\end{equation}
which is fully faithful and the identity on objects, while
$\ex^\infty(\mathcal P)$ is fibrant: we define
$\ex^\infty(\mathcal P)$ by $\ob \ex^\infty(\mathcal P)=\ob \mathcal P$
on objects, while, for any $(n+1)$-tuple of objects $(x_1,\ldots,x_n,x)$
of $\mathcal P$, we take
\begin{equation}
 \ex^\infty(\mathcal P)(x_1,\ldots,x_n;x)= \ex^\infty(\mathcal P(x_1,\ldots,x_n;x))\, ,
\end{equation}
where, one the right hand side, we apply Kan's $\ex^\infty$ functor to the
space of operations of $\mathcal P$. This defines a simplicial operad
because Kan's $\ex^\infty$ functor preserves finite products.
In fact, more generally, this functor preserves finite limits as well as filtered
colimits, so that the functor $\mathcal P\longmapsto\ex^\infty(\mathcal P)$
has the same properties. Moreover, Kan's $\ex^\infty$ functor preserves
Kan fibrations, so that the induced functor on simplicial operads
preserves local fibrations as well as isofibrations (for the latter property, note
that $\pi_0(\mathcal P)\simeq\pi_0(\ex^\infty(\mathcal P))$).
\end{rem}

\begin{thm}\label{thm:cmcfixedcoloursoper}
The category $\soper_C$ is endowed with a right proper cofibrantly generated
model category structure whose weak equivalences (fibrations)
are the morphisms of $\soper_C$ which are fully faithful (which are local fibrations, respectively).
\end{thm}

\begin{proof}
As recalled in the preceding remark,
the usual Quillen model category structure on the category of simplicial sets
has a finite limit preserving fibrant replacement functor provided
by Kan's $\ex^\infty$ functor. Moreover the canonical interval $\Delta[1]$
is naturally endowed with a structure of cocommutative co-unital comonoid
(the comultiplication is the diagonal $\Delta[1]\to\Delta[1]\times\Delta[1]$,
while the co-unit is the unique map $\Delta[1]\to\Delta[0]$),
so that the existence of this model category structure on the category
$\soper_C$ follows straight away from \cite[Theorem 2.1]{BMOp1}.
The right properness easily follows from the analogous property
in the usual model category of simplicial sets.
\end{proof}

\begin{rem}
A morphism in $\soper_C$ is a trivial fibration if and only if it is a local trivial
fibration in the sense of \ref{def:mainclassesofmapssimpoper}.
\end{rem}

\begin{rem}
The model structure of Theorem \ref{thm:cmcfixedcoloursoper} can be constructed by
more elementary means: the category \smash{$\soper_C$} is the category of simplicial objects
in the category \smash{$\oper_C$} of operads with fixed set of object $C$, and one can check that
\smash{$\underline A=\oper_C$} satisfies the hypotheses of \cite[Chapter II, Section 4, Theorem 4]{quillen}.
\end{rem}

\begin{paragr}\label{paragr:changeofcolours}
Consider two sets $C$ and $D$ as well as a map $f:C\to D$.
This map induces an adjunction
\begin{equation}\label{eq:changeofcolours}
f_!:\soper_C\rightleftarrows\soper_D:f^*
\end{equation}
in which the right adjoint sends a simplicial operad $\mathcal P$ with set of objects $D$
to the simplicial operad $f^*\mathcal P$ with set of objects $C$ defined by the formula
\begin{equation}
f^*\mathcal P(c_1,\ldots,c_n;c)=\mathcal P(f(c_1),\ldots,f(c_n);f(c))
\end{equation}
for any integer $n\geq 0$ and any $(n+1)$-tuple $(c_1,\ldots,c_n,c)$ in $C^{n+1}$.

The following statement is then obvious.
\end{paragr}

\begin{prop}\label{prop:changeofcolours}
The adjunction \eqref{eq:changeofcolours} is a Quillen pair with respect to the
model category structures of Theorem \ref{thm:cmcfixedcoloursoper}.
Moreover, the functor $f^*$ preserves weak equivalences.
\end{prop}

\begin{rem}
Consider an arbitrary morphism of simplicial operads $u:\mathcal P\to \mathcal Q$.
If, with a slight abuse of notation, we still denote by $u:\ob\mathcal P\to\ob \mathcal Q$
the associated map on objects, we obtain a natural factorization of the morphism of simplicial
operads $u$ of the form
\begin{equation}
\mathcal P\to u^*\mathcal Q \to \mathcal Q
\end{equation}
in which the first map $\mathcal P\to u^*\mathcal Q$ is bijective on objects, while the second
one, $u^*\mathcal Q \to \mathcal Q$, is degreewise fully faithful (in the strong sense of \ref{def:folkcmcoperads}).

We see in particular that the morphism $u$ is fully faithful (a local fibration) in the sense of \ref{def:mainclassesofmapssimpoper}
if and only if the map $\mathcal P\to u^*\mathcal Q$
is a weak equivalence (a fibration, respectively) of the model structure of Theorem \ref{thm:cmcfixedcoloursoper}
(with $C=\ob \mathcal P$).
\end{rem}

\begin{rem}\label{rem:fromoperadstocategories}
The category $\cat$ of small categories can be seen as the full subcategory
of the category of operads whose objects are the operads $\mathcal P$ such that
$\mathcal P(x_1,\ldots,x_n;x)=\varnothing$ for any non-negative integer $n\neq 1$
and any $(n+1)$-tuple of objects $(x_1,\ldots,x_n,x)$ of $\mathcal P$.
We thus get an adjunction
\begin{equation}\label{eq:adjcatoper}
j_!:\cat\rightleftarrows\oper:j^*
\end{equation}
where $j_!$ denotes the inclusion. One can understand this adjunction as follows.
Let $\eta$ be the terminal category, seen as an operad. Then there is a canonical
identification $\cat=\oper/\eta$, so that the inclusion functor $j_!$ corresponds
to the forgetful functor $\oper/\eta\to\oper$, while its right adjoint
is the functor $\mathcal P\longmapsto \eta\times\mathcal P$.
In particular, the model category structure of Theorem \ref{thm:folkcmcoperads}
induces a model category structure on $\oper/\eta=\cat$ which is nothing but
the canonical model category structure (whose weak equivalences are the
equivalences of categories and whose cofibrations are the functors which
are injective on objects). The adjunction \eqref{eq:adjcatoper} extends naturally
to an adjunction relating simplicial categories and simplicial operads.
\begin{equation}\label{eq:adjscatsoper}
j_!:\scat\rightleftarrows\soper:j^*
\end{equation}
A similar pattern appears if we fix a set of objects $C$: the category $\scat_C$
of simplicial categories with a fixed set of objects $C$ can be seen as a full subcategory
of the category \smash{$\soper_C$}. The terminal object of $\scat_C$ is the contractible
groupoid whose set of objects is $C$, which we will denote by $EC$. We have
a natural identification \smash{$\scat_C=\soper_C/j_! EC$}, from which the model
structure of Theorem \ref{thm:cmcfixedcoloursoper} induces the existence
of a right proper model category structure on $\scat_C$. The latter is in fact
the model category structure of Dwyer \& Kan~\cite[Proposition 7.2]{DK}.
\end{rem}

The main purpose of this section is to prove the following theorem, which can be seen
as an extension to operads of the model category structure on $\scat$ constructed by
J.~Bergner~\cite{bergnera}.

\begin{thm}\label{thm:cmcsoper}
The category $\soper$ of simplicial operads is endowed with a cofibrantly generated
right proper
model category structure in which a map
$u:\mathcal P\to \mathcal Q$ is a weak equivalence (a fibration) if and only if it is
fully faithful and essentially surjective (an isofibration, respectively)
in the sense of \ref{def:mainclassesofmapssimpoper}.
\end{thm}

The proof of this theorem will require quite a few steps.
We will now develop the preparatory results needed for the proof,
while the formal proof will be given at the end of this section (page \pageref{proofofcmcsoper}).

\begin{paragr}\label{paragr:defcofandtrivcof}
We introduce several classes of maps
in the category of simplicial operads. For this purpose, we will need the following construction.
Given a simplicial set $X$ and an integer $n\geq 0$, we denote by $C_n[X]$ the
free simplicial operad with $n$-ary operations decorated by the simplicial set $X$.
In other words, the simplicial operad $C_n[X]$ has objects $0,1,\ldots,n$, and, for each
simplicial degree $m$, the only non-identity operations of the operad $C_n[X]_m$ are given by the set
\begin{equation}\label{eq:defreprhomoper1}
C_n[X]_m(1,\ldots,n;0)=X_m\, .
\end{equation}
Another way of defining $C_n[X]$ is by saying that it represents the functor
\begin{equation}\label{eq:defreprhomoper2}
\op{\soper}\to\set \ , \quad \mathcal P\longmapsto
\coprod_{(x_1,\ldots,x_n,x)\in\ob \mathcal P^{n+1}}\Hom_{\sset}(X,\mathcal P(x_1,\ldots,x_n;x))\, .
\end{equation}
This construction is certainly functorial in $X$.

Now we can go back to defining our classes of maps in $\soper$.
The class $\mathcal C$ contains:
\begin{itemize}
\item[(C1)] the map $\varnothing\to\eta$ (where $\eta$ stands for the
terminal category, seen as an operad);
\item[(C2)] the maps $C_n[\partial\Delta[m]]\to C_n[\Delta[m]]$,
for any integer $n\geq 0$
and any boundary inclusion $\partial\Delta[m]\to\Delta[m]$, $m\geq 0$.
\end{itemize}
The class $\overline{\mathcal C}$ is the saturation of $\mathcal C$ (i.e. the closure under
pushout, transfinite composition, and retracts). The class $\mathcal A$ contains
\begin{itemize}
\item[(A1)] the maps of shape $\eta\to H$ which are fully faithful and essentially surjective
(in the sense of \ref{def:mainclassesofmapssimpoper}), while $H$ is a category
with set of objects $\{0,1\}$, with the property that $H_n$ is countable for each integer $n\geq 0$,
and that $H$ is cofibrant, seen as object of $\soper_{\{0,1\}}$ (or equivalently, of $\scat_{\{0,1\}}$).
\item[(A2)] the maps $C_n[\Lambda^k[m]]\to C_n[\Delta[m]]$
by any horn inclusion $\Lambda^k[m]\to\Delta[m]$, $m\geq 1$, $0\leq k \leq m$,
and for any integer $n\geq 0$.
\end{itemize}
As before, the class $\overline{\mathcal A}$ is the saturation of $\mathcal A$.

The following lemma shows that, for any monomorphism (trivial cofibration)
of simplicial sets $X\to Y$, and for any integer $n\geq 0$, the induced morphism
of simplicial operads
$C_n[X]\to C_n[Y]$ is in $\overline{\mathcal C}$ (in $\overline{\mathcal A}$, respectively).
\end{paragr}

\begin{lemma}\label{lemma:localRLP}
Let $i:X\to Y$ be a morphism of simplicial sets, and $u:\mathcal P\to \mathcal Q$ a morphism
of simplicial operads. The following properties are equivalent.
\begin{itemize}
\item[(i)] For any integer $n\geq 0$ and any $(n+1)$-tuple of objects $(x_1,\ldots,x_n,x)$ of $\mathcal P$,
the morphism $\mathcal P(x_1,\ldots,x_n;x)\to \mathcal Q(u(x_1),\ldots,u(x_n);u(x))$
has the right lifting property with respect to $X\to Y$.
\item[(ii)] The morphism $u$ has the right lifting property with respect to any map
of shape $C_n[i]:C_n[X]\to C_n[Y]$, for any integer $n\geq 0$.
\end{itemize}
\end{lemma}

\begin{proof}
This follows by a standard argument, once we noticed that $C_n[X]$ represents the
presheaf \eqref{eq:defreprhomoper2}.
\end{proof}

\begin{paragr}
As will be shown later in the proof of Theorem \ref{thm:cmcsoper},
the class $\overline{\mathcal C}$ is the class of cofibrations in $\soper$,
while $\overline{\mathcal A}$ is the class of trivial cofibrations.
Propositions \ref{prop:trivfibrationsoper} and \ref{prop:fibrationsoper} clarify this point.
\end{paragr}

\begin{prop}\label{prop:trivfibrationsoper}
A morphism of simplicial operads has the right lifting property with respect to
the class $\overline{\mathcal C}$ if and only if it is a local trivial fibration which is
surjective on objects.
\end{prop}

\begin{proof}
The right lifting property with respect to the maps of shape (C1) is equivalent to the surjectivity on objects,
while, by virtue of Lemma \ref{lemma:localRLP}, the right lifting property with respect to maps
of shape (C2) corresponds to the property of being a local trivial fibration.
\end{proof}

\begin{cor}\label{cor:trivcofintocof}
We have $\overline{\mathcal A}\subset\overline{\mathcal C}$.
\end{cor}

\begin{proof}
The preceding proposition implies that any morphism with the right
lifting property with respect to $\overline{\mathcal C}$ also has
the right lifting property with respect to maps of type (A2).
It also implies, by examination of lifting properties, that,
for any set $C$, the image of any cofibration of \smash{$\soper_C$} in $\soper$
is in $\overline{\mathcal C}$. As the inclusion
$\{0\}\to \{0,1\}=\{0\}\amalg\{1\}$
is certainly in $\overline{\mathcal C}$, this implies that any map of shape
(A1) is in $\overline{\mathcal C}$. Therefore,
$\mathcal A\subset \overline{\mathcal C}$, whence the result.
\end{proof}

\begin{prop}\label{prop:fibrationsoper}
A morphism of simplicial operads is an isofibration if and only if it has the
right lifting property with respect to the class $\overline{\mathcal A}$.
\end{prop}

\begin{proof}
Lemma \ref{lemma:localRLP} implies that
the right lifting property with respect to maps of shape (A2) is equivalent to the property of being
a local fibration. Therefore, it is sufficient to prove that a local fibration $u:\mathcal P \to \mathcal Q$
has the right lifting property with respect to maps of shape (A1) if and only if the morphism
of operads $\pi_0(u):\pi_0(\mathcal P)\to\pi_0(\mathcal Q)$ is an isofibration.
Note that, if $u$ is a local fibration, so is the simplicial functor $j^*u:j^*\mathcal P\to j^*\mathcal Q$.
Moreover, the simplicial functor $j^*u$ is an isofibration if and only if $u$ has this property
(just remark that the functor $j^*$ commutes with $\pi_0$).
In other words, it is sufficient to prove this lemma in the case when both $\mathcal P$
and $\mathcal Q$ are in fact simplicial categories. But then, this is already known; see
\cite[Propositions 2.3 and 2.5]{bergnera}.
\end{proof}

\begin{prop}\label{prop:trivfibloctrivfib}
A morphism of simplicial operads is both a weak equivalence and an isofibration
if and only if it has the right lifting property with respect to the class $\overline{\mathcal C}$.
\end{prop}

\begin{proof}
Let $u:\mathcal P \to \mathcal Q$ a morphism of simplicial operads. If $u$ is surjective on objects
and is a local trivial fibration, then it is surely a local fibration. Furthermore, the morphism
of operads $\pi_0(u)$ is then fully faithful and surjective on objects, which implies that it
is a trivial fibration for the model category structure of Theorem \ref{thm:folkcmcoperads}.
In particular, the morphism $\pi_0(u)$ is an isofibration, whence $u$ is an isofibration.
Conversely, if $u$ is an isofibration and a weak equivalence, then $u$ is obviously
a local trivial fibration, while $\pi_0(u)$ is a weak equivalence and an isofibration
for the model category structure of Theorem \ref{thm:folkcmcoperads}. Therefore,
$\pi_0(u)$ must be surjective on objects, which implies that $u$ has the same property.
We conclude the proof with Proposition \ref{prop:trivfibrationsoper}.
\end{proof}

\begin{paragr}\label{paragr:defexponentialsoper}
If $\mathcal P$ is a simplicial operad and if $X$ is a simplicial set, we denote by $\mathcal P^X$
the simplicial operads defined by $\ob\mathcal P^X=\ob\mathcal P$ and by
$$\mathcal P^X (x_1,\ldots,x_n;x)=\mathcal P (x_1,\ldots,x_n;x)^X$$
(where, for two simplicial sets $X$ and $Y$, we denote by $Y^X$ the internal Hom of maps from $X$ to $Y$).
\end{paragr}

\begin{prop}\label{prop:2outof3soper}\label{prop:wesoperretract}
In a commutative triangle of $\soper$, if two of the maps are weak equivalences, then so is
the third. Furthermore, the class of weak equivalences of simplicial operads is closed under retracts.
\end{prop}

\begin{proof}
Using the canonical fibrant replacement functor given by Remark \ref{rem:nicefibrantreplacementsoper},
it is sufficient to prove this proposition in the full subcategory of fibrant simplicial operads.
As, for any simplicial set $X$ and any simplicial operad $\mathcal P$, the
morphism $X\to\Delta[0]$ induces a canonical morphism of
simplicial operads $\mathcal P\to\mathcal P^X$
which is bijective on objects, a morphism of simplicial operads $\mathcal P\to\mathcal Q$ is
essentially surjective if and only if, for any simplicial set $X$, the morphism of operads
$\pi_0(\mathcal P^X)\to\pi_0(\mathcal Q^X)$ is essentially surjective.
On the other hand, the Yoneda Lemma applied to the homotopy category of Kan complexes implies that
a morphism of fibrant simplicial operads $\mathcal P\to\mathcal Q$ is fully faithful if and only
if, for any simplicial set $X$, the morphism of operads $\pi_0(\mathcal P^X)\to\pi_0(\mathcal Q^X)$
is fully faithful. In conclusion, a morphism of fibrant simplicial operads $\mathcal P\to\mathcal Q$ is
a weak equivalence if and only if, for any simplicial set $X$, the morphism of operads
$\pi_0(\mathcal P^X)\to\pi_0(\mathcal Q^X)$ is a weak equivalence of the model category given by
Theorem \ref{thm:folkcmcoperads}. This immediately implies the proposition.
\end{proof}

\begin{prop}\label{prop:wesoperfiltcolim}
The class of morphisms of $\soper$ which are  fully faithful (essentially
surjective, respectively) is closed under filtered colimits.
\end{prop}

\begin{proof}
Let us consider first the case of essentially surjective maps.
The functor $\pi_0:\soper\to\oper$ commutes with (filtered) colimits, so that
it is suficient to check the analogous property for operads, which is easy.

Consider now a filtered diagam of fully faithful morphisms of simplicial operads
$u_i:\mathcal P_i\to \mathcal Q_i$, and let us denote by $u:\mathcal P\to\mathcal Q$
the colimit of these. Consider an integer $n\geq 0$ and an $(n+1)$-tuple
$(x_1,\ldots,x_n,x)$ of objects of $\mathcal P$. We may assume that it comes
from a compatible family of $(n+1)$-tuples $(x^i_{1},\ldots,x^i_{n},x^i)$ of objects of
$\mathcal P_i$ for each $i$ (by replacing the indexing category by a cofinal
subcategory). We then have natural identifications of the form
$$\varinjlim_i \mathcal P_i(x^i_{1},\ldots,x^i_{n};x^i)\simeq\mathcal P(x_1,\ldots,x_n;x)\, ,$$
and, similarly,
$$\varinjlim_i \mathcal Q_i(u_i(x^i_{1}),\ldots,u_i(x^i_{n});u_i(x^i))\simeq\mathcal Q(u(x_1),\ldots,u(x_n);u(x))\, .$$
The fact that $u$ is fully faithful thus comes down to the fact that the class of weak equivalences
of simplicial sets is closed under filtered colimits.
\end{proof}

\begin{prop}\label{prop:trivcofibsoper}
The class $\overline{\mathcal A}$ is contained in the class of weak
equivalences.
\end{prop}

\noindent The proof of this proposition will go through the following chain
of four lemmas.

\begin{lemma}\label{lemma:injectivepushoutsoper}
Consider the following pushout square in the category of simplicial operads.
$$\xymatrix{
\mathcal P\ar[r]^f\ar[d]_u&\mathcal P'\ar[d]^{u'}\\
\mathcal Q\ar[r]_g&\mathcal Q'
}$$
If $u$ is in $\overline{\mathcal C}$, is bijective on objects, and is
a weak equivalence, and if $f$ is injective on objects, then $u'$
is a weak equivalence.
\end{lemma}

\begin{proof}
We may assume that $u$ induces the identity on objects, so that
$u$ must be a trivial cofibration in the model category structure
of Theorem \ref{thm:cmcfixedcoloursoper} for $C=\ob\mathcal P$.
Write $C'=\ob\mathcal P'$, and $C'=B\amalg C$ (in such a way that $f$
is the coproduct inclusion). For a simplicial operad $\mathcal P$,
we define $\mathcal P\amalg B$ to be the simplicial operad
$$\mathcal P\amalg B=\mathcal P\amalg \Big(\coprod_{b\in B}\eta\Big) \, .$$
The pushout of the lemma can then be seen as the composition of two pushouts
of the following shape.
$$\xymatrix{
\mathcal P\ar[r]\ar[d]&\mathcal P\amalg B\ar[r]^{f'}\ar[d]_{u\amalg B}&\mathcal P'\ar[d]^{u'}\\
\mathcal Q\ar[r]&\mathcal Q\amalg B\ar[r]_{g'}&\mathcal Q'
}$$
Now, the right hand square is a pushout in $\soper_{C'}$, and,
as $u$ is a trivial cofibration of \smash{$\soper_C$}, the map $u\amalg B$
must be a trivial cofibration of \smash{$\soper_{C'}$}. Therefore, the morphism $u'$
is a trivial cofibration in \smash{$\soper_{C'}$}.
\end{proof}

\begin{lemma}\label{lemma:decoratedcorollastrivcof}
Let $\varphi:\{0,\ldots,n\}\to S$ be a surjection ($n\geq 0$).
Then, for any morphism of simplicial sets $X\to Y$, the commutative square
$$\xymatrix{
C_n[X]\ar[r]\ar[d]&C_n[Y]\ar[d]\\
\varphi_! C_n[X]\ar[r]&\varphi_! C_n[Y]
}$$
is a pushout in $\soper$. Moreover, if $X\to Y$ is a trivial cofibration of
simplicial sets, then the map $\varphi_! C_n[X]\to \varphi_! C_n[Y]$ is a trivial
cofibration in the model category structure on $\soper_S$ given by Theorem \ref{thm:cmcfixedcoloursoper}.
\end{lemma}

\begin{proof}
For a simplicial set $X$,
the object $\varphi_!C_n[X]$ represents the presheaf
$$\op{\soper}\to \set \ , \quad
\mathcal P\longmapsto \coprod_{f\in \Hom(S,\ob\mathcal P)}\mathcal P(f\varphi(1),\ldots,f\varphi(n);f\varphi(0))\, .$$
We deduce from this description of the functor $\varphi_!C_n[-]$ that, for any simplicial operad $\mathcal P$,
the commutative square
$$\xymatrix{
\Hom_{\soper}(\varphi_!C_n[Y],\mathcal P)\ar[r]\ar[d]&\Hom_{\soper}(\varphi_!C_n[X],\mathcal P)\ar[d]\\
\Hom_{\soper}(C_n[Y],\mathcal P)\ar[r]&\Hom_{\soper}(C_n[X],\mathcal P)
}$$
is cartesian; or, in other words, that the square of the lemma is cocartesian.

To prove the last assertion, as the functor $\varphi_!$ is a left Quillen functor (Proposition \ref{prop:changeofcolours}),
we may assume that $\varphi$ is the identity. But then, if $X\to Y$ is a trivial cofibration, it follows from
Lemma \ref{lemma:localRLP} that the morphism
$C_n[X]\to C_n[Y]$ has the right lifting property with respect to the fibrations of the model
structure given by Theorem \ref{thm:cmcfixedcoloursoper} for $C=\{0,\ldots,n\}$. Therefore, it must be a
trivial cofibration.
\end{proof}

\begin{lemma}\label{lemma:restobjpreservecofibrants}
Let $f:C\to D$ be an inclusion of sets. Then the pullback on objects
functor $f^*:\scat_D\to\scat_C$ preserves cofibrant objects.
\end{lemma}

\begin{proof}
An object of $\scat_C$ is cofibrant if and only if it is a retract
of a simplicial category $A$ such that, for each integer $n\geq 0$,
$A_n$ is isomorphic to the free category on a directed graph (whose set of
objects is $C$), such that the degeneracies preserves generators;
see \cite[7.6]{DK}. Therefore, it is sufficient to prove that the functor $f^*$
preserves free categories on graphs in a functorial way.
Let $G$ be a directed graph
with $D$ as set of objects. Let us denote by $f^*G$
the graph whose set of objects is $C$, and whose set of edges $c\to c'$
consists of the edges $c\to c'$ in $G$ together with the formal strings of edges of $G$ of shape
$$c\to d_1\to \cdots \to d_n\to c'\ , \quad n\geq 1\, ,$$
with $d_i\in D- C$ for $1\leq i \leq n$. This construction is functorial in $G$, at least
for morphisms of graphs which are the identity on objects.
One checks that the image by $f^*$ of the free category generated by $G$ is naturally
isomorphic to the free category generated by $f^*G$.
\end{proof}

\begin{lemma}\label{lemma:pushoutstrictfullyfaithful}
Let $u:K\to H$ be a full embedding of small categories, in which the set of objects of $H$ is $\{0,1\}$,
while $K$ only has one object $0$ (such that $u(0)=0$).
Consider a pushout of shape
$$\xymatrix{
j_!(K)\ar[r]^f\ar[d]_{j_!(u)}& P\ar[d]^v\\
j_!(H)\ar[r]^g&Q
}$$
in the category $\oper$ of operads. Then the map $v$ is fully faithful.
\end{lemma}

\begin{proof}
The proof of this lemma is elementary but tedious: it consists
of an explicit description of the operad $Q$.

Suppose that $P$ is an operad with set of objects $C$, and let $c_0=f(0)$.
Then the set of objects of the operad $Q$ is the disjoint union of the set $C$ and of an element which we will
denote by $t$. Up to permutations of the input objects, the $n$-ary operations in $Q$ (with $n\geq 0$)
are of five kinds:
\begin{itemize}
\item[(i)] $c_1,\ldots,c_n\to c$ (with $c_i$ and $c$ in $C$);
\item[(ii)] $c_1,\ldots,c_n\to t$ (with $c_i$ in $C$);
\item[(iii)] $c_1,\ldots,c_a,t,\ldots,t\to t$ (with $c_i$ in $C$
and a sequence of $b$ occurrences of $t$, $b\geq 1$, $a+b=n$);
\item[(iv)] $c_1,\ldots,c_a,t,\ldots,t\to c_0$ (with $c_i$ in $C$
and a sequence of $b$ occurrences of $t$, $b\geq 1$, $a+b=n$);
\item[(v)] $t\to t$ (like case (iii), with $n=b=1$ and $a=0$).
\end{itemize}

For (i) we take the same operations as in $P$.

For (ii), we take the operations to be represented by pairs $(h,p)$, written as
$$h\specialbullet p$$
where $h\in H(0,1)$ and $p\in P(c_1,\ldots,c_n;c_0)$; two such operations are identified
by the rule
$$h u(k)\specialbullet p=h\specialbullet f(k)p\quad\text{for any $k\in K(0,0)$.}$$

For (iii), we take operations to be represented by $(b+2)$-tuples, written as
$$h\specialbullet p\specialbullet (h_1,\ldots,h_b)$$
where $h\in H(0,1)$, $(h_1,\ldots,h_b)\in H(0,1)^b$, and
$p\in P(c_1,\ldots,c_a,c_0,\ldots,c_0;c_0)$ (with $b$ occurrences of $c_0$ before the semi-colon);
two such operations are identified according to the rules
$$hu(k)\specialbullet p\specialbullet (h_1,\ldots,h_b)=h\specialbullet f(k)p\specialbullet (h_1,\ldots,h_b)$$
for any $k\in K(0,0)$, and
$$h\specialbullet p\specialbullet (h_1,\ldots,u(k)h_i,\ldots,h_b)=h\specialbullet \, p\, \circ_i f(k)\specialbullet (h_1,\ldots,h_b)$$
for any $k\in K(0,0)$ and $1\leq i\leq b$.

For (iv), we take the operations of the form
$$p\specialbullet (h_1,\ldots,h_b)$$
with $(h_1,\ldots,h_b)\in H(1,0)^b$, and $p\in P(c_1,\ldots,c_a,c_0,\ldots,c_0;c_0)$;
two such operations are identified by the rule
$$p\specialbullet (h_1,\ldots,u(k)h_i,\ldots,h_b)=p\, \circ_i f(k)\specialbullet (h_1,\ldots,h_b)$$
for any $k\in K(0,0)$ and $1\leq i\leq b$.

Finally, for (v), we take the elements of $H(1,1)$ as operations in $Q(t,t)$.

Composition of such operations is defined in the evident way (modulo permutations of arguments),
but uses that the functor $K\to H$ is fully faithful in an essential way. For example, if
$$\varphi=h\specialbullet p\specialbullet (h_1,\ldots,h_b) \ \text{and} \
\psi=p'\specialbullet (h'_1,\ldots,h'_{b'})\, ,$$
are operations of type (iii) and (iv) respectively, with $b'\geq 2$, then
$$\psi\circ_{n+1}\varphi=(p'\circ_{n+1} f(k)p)\specialbullet (h_1,\ldots,h_b,h'_2,\ldots,h'_{b'})$$
where $k$ is the unique operation such that $u(k)=h'_1h$.

Finally the operad $Q$ (defined above) fits into a commutative diagram of operads
$$\xymatrix{
K\ar[r]^f\ar[d]_{u}& P\ar[d]^v\\
H\ar[r]^g&Q
}$$
(where we identify $K$ and $H$ with their images under the functor $j_!$), where
the map $v$ sends operations in $P$ to operations of type (i) in $Q$, and, for $h\in H(i,j)$,
the value of $g(h)$ is defined as follows:
$$g(h)=\begin{cases}
h\in Q(c_0;c_0)&\text{if $i=j=0$ (type (i)),}\\
h\specialbullet 1_{c_0}\in Q(c_0;t)&\text{if $i=0$, $j=1$ (type (ii)),}\\
1_{c_0}\specialbullet h\in Q(t;c_0)&\text{if $i=1$, $j=0$ (type (iv)),}\\
h\in Q(t;t)&\text{if $i=j=1$ (type (v)).}\\
\end{cases}$$
The morphism $v$ is fully faithful by construction.
It remains to check that the commutative square above is a pushout.
Let $R$ be another operad, and consider a commutative square as below.
$$\xymatrix{
K\ar[r]^f\ar[d]_{u}& P\ar[d]^\alpha\\
H\ar[r]^\beta&R
}$$
One defines a morphism $\gamma:Q\to R$ on objects by $\gamma(c)=\alpha(c)$
and $\gamma(t)=\beta(1)$, and on operations of type (i)--(v) respectively by:
\begin{itemize}
\item[(i)] $\gamma(p)=\alpha(p)$;
\item[(ii)] $\gamma(h\specialbullet p)=\beta(h)\circ \alpha(p)$;
\item[(iii)] $\gamma(h\specialbullet p\specialbullet (h_1,\ldots,h_b))=
\beta(h)\circ\alpha(p)\circ (\beta(h_1),\ldots,\beta(h_b))$ (this last circle
denotes the operad composition given by inserting $\beta(h_1),\ldots,\beta(h_b)$
in the entries $a+1,\ldots,n$ of the operation $\alpha(p)$ --- we could
also have written $\beta(h)(\alpha(p)(1_{\alpha(c_1)},\ldots,1_{\alpha(c_a)},\beta(h_1),\ldots,\beta(h_b)))$);
\item[(iv)] $\gamma(p\specialbullet (h_1,\ldots,h_b))=\alpha(p)\circ (\beta(h_1),\ldots,\beta(h_b))$;
\item[(v)] $\gamma(h)=g(h)$.
\end{itemize}
One checks that $\gamma$ is a well defined map of operads, and that it is the unique one
such that $\gamma g=\beta$ and $\gamma v= \alpha$.
\end{proof}

We now are ready to present the

\begin{proof}[Proof of Proposition \ref{prop:trivcofibsoper}]
Note that the class of morphisms of simplicial operads which are fully faithful and
essentially surjective is closed under retracts and under small filtering colimits;
see Propositions \ref{prop:wesoperretract} and \ref{prop:wesoperfiltcolim}.
As the class $\overline{\mathcal A}$ is the saturation
of a small set of maps (up to completion by isomorphisms), namely $\mathcal A$,
we can apply the small object argument and see that any map in \smash{$\overline{\mathcal A}$}
is a retract of a transfinite composition of pushouts of maps of type (A1) or (A2).
Therefore, it is sufficient to prove that any pushout of a map of type (A1) or (A2) is
fully faithful and essentially surjective.

Let us consider first a trivial cofibration of simplicial sets $X\to Y$, a non-negative integer $n$,
as well as a pushout of the following form in the category of simplicial operads.
$$\xymatrix{
C_n[X]\ar[r]^f\ar[d]_a&\mathcal P\ar[d]^b\\
C_n[Y]\ar[r]^g&\mathcal Q
}$$
At the level of objects, the map $f$ factors as a surjection
$\varphi:\{0,\ldots,n\}\to S$ followed by an injection $S\to\ob\mathcal P$.
By virtue of Lemma \ref{lemma:decoratedcorollastrivcof}, the pushout square above
decomposes into two pushout squares
$$\xymatrix{
C_n[X]\ar[r]\ar[d]_a&\varphi_!C_n[X]\ar[r]\ar[d]_{\varphi_!(a)}&\mathcal P\ar[d]^b\\
C_n[Y]\ar[r]&\varphi_!C_n[Y]\ar[r]&\mathcal Q
}$$
in which $\varphi_!(a)$ is a trivial cofibration of $\soper_S$. By applying Lemma \ref{lemma:injectivepushoutsoper}
to the right hand square, we conclude that the map $b$ is a weak equivalence.

Now, consider a weak equivalence of shape $a:\eta=\{0\}\to H$,
where $H$ is a cofibrant object of $\soper_{\{0,1\}}$,
as well as a pushout square of the following form.
$$\xymatrix{
\eta\ar[r]^a\ar[d]_f&H\ar[d]^g\\
\mathcal P\ar[r]^b& \mathcal Q
}$$
Let $K=a^*H$, and decompose the pushout above into two pushouts as below.
$$\xymatrix{
\eta\ar[r]^{a_1}\ar[d]_f&K\ar[r]^{a_2}\ar[d]^h&H\ar[d]^g\\
\mathcal P\ar[r]^{b_1}&\mathcal R\ar[r]^{b_2}& \mathcal Q
}$$
Lemma \ref{lemma:restobjpreservecofibrants} shows that the morphism $a_1$ is in $\overline{\mathcal C}$,
so that we can apply Lemma \ref{lemma:injectivepushoutsoper} to the left hand square,
from which we deduce that the map $b_1$ is a weak equivalence. To finish the proof of Proposition \ref{prop:trivcofibsoper},
it is thus sufficient to prove that the map $b_2$ is a weak equivalence.
Note that the map $a_2$ is a fully faithful morphism of operads in each simplicial degree.
Therefore, by virtue of Lemma \ref{lemma:pushoutstrictfullyfaithful},
the map $b_2$ has the same property. In particular, the map $b_2$ is
fully faithful in the sense of \ref{def:mainclassesofmapssimpoper}.
Moreover, the map $b$ is essentially surjective: $\pi_0(b)$ is a pushout of $\pi_0(a)$,
which implies that it is an equivalence of operads (as the pushout of a trivial
cofibration of the model structure of Theorem \ref{thm:folkcmcoperads}),
whence an essentially surjective map.
But then, the map $b_2$ must be essentially surjective as well, whence a weak equivalence.
\end{proof}

\begin{proof}[Proof of Theorem \ref{thm:cmcsoper}]\label{proofofcmcsoper}
We will prove a slightly more precise result: namely that the class
of cofibrations is $\overline{\mathcal C}$ and the class of trivial
cofibrations is $\overline{\mathcal A}$; see \ref{paragr:defcofandtrivcof}.
We already know that these classes are generated by small sets of maps, so that we can
use the small object argument to prove the existence of factorizations.
By virtue of Propositions \ref{prop:fibrationsoper}, \ref{prop:trivfibloctrivfib}, \ref{prop:2outof3soper}
and \ref{prop:trivcofibsoper}, the conditions
of \cite[Theorem 2.1.19]{Ho} are verified, which gives the existence of the
expected model category structure on $\soper$.
The right properness property comes from the fact that Kan's $\ex^\infty$
functor induces a fibrant replacement functor in $\soper$ which
commutes with finite limits and which preserves fibrations (see Remark \ref{rem:nicefibrantreplacementsoper}).
\end{proof}

\section{Dendroidal sets and $\infty$-operads}\label{sect:2}

In this section, nothing new happens: we just recall, for the convenience of the reader,
the results of \cite{dend3,dend6} on the homotopy theory of $\infty$-operads
needed in this paper.

\begin{paragr}
The category of trees $\trees$ is the full subcategory of the category
of operads whose objects are trees (by which we mean non empty finite rooted trees),
seen as operads: if $T$ is a tree the corresponding operad is the free operad
whose objects are the edges of $T$, generated by an operation
$(e_1,\ldots,e_n)\to e$ for each vertex $v$ of $T$, where $e_1,\ldots,e_n$
are the input edges at $v$ ($n\geq 0$), while $e$ is the output edge at $v$.
One proves that any morphism of $\trees$ can be factored as a split epimorphism
followed by a monomorphism.

The category $\dset$ of \emph{dendroidal sets} is the category of presheaves of sets on the category
of trees $\trees$. If $T$ is a tree, we denote by
$$\Omega[T]=\Hom_\trees(-,T)$$
the corresponding representable presheaf.
The \emph{boundary} inclusion $\partial\Omega[T]\to\Omega[T]$ is defined as the union
of the images of all the monomorphisms $\Omega[S]\to\Omega[T]$ which are not isomorphisms.
If $e$ is an \emph{inner edge} of $T$ (i.e. an edge which connects two vertices in $T$), and if
$T/e$ denotes the tree obtained from $T$ by contracting $e$, then there is a unique
monomorphism $\partial_e:T/e\to T$ in $\trees$, called the \emph{inner face} associated to $e$,
such that $e$ is not in the image of $\partial_e$.
The corresponding \emph{inner horn} inclusion $\Lambda^e[T]\to\Omega[T]$ is the union of the
images of all the monomorphisms $\Omega[S]\to\Omega[T]$ which are not isomorphisms and
which do not factor through $\partial_e$.
\end{paragr}

\begin{paragr}
Let $0$ be the tree without any vertices (as an operad, this is the
terminal category $\eta$). By a slight abuse
of notation, we write $\Omega[0]=\eta$.
Then the comma category $\Omega/0$
is isomorphic to the category $\Delta$ of non-empty totally ordered finite sets
(in $\Omega$, these correspond to linear trees).
In particular, the category $\dset/\eta$ can be identified with the category $\sset$
of simplicial sets, and one has an adjunction
\begin{equation}\label{eq:adjssetdset}
\begin{split}
i_!:\sset\rightleftarrows\dset:i^*
\end{split}
\end{equation}
in which the left adjoint has the property of being fully faithful: one checks that $\eta$
is a subobject of the terminal object of $\dset$.
This adjunction extends the adjunction \eqref{eq:adjcatoper}
in the following sense. The inclusion $\Delta\subset\cat$
extends to an adjunction
\begin{equation}\label{eq:nerfcat}
\begin{split}
\tau:\sset\rightleftarrows\cat:\nerf
\end{split}
\end{equation}
and, similarly the inclusion $\Omega\subset\oper$ extends to an adjunction
\begin{equation}\label{eq:nerfdend}
\begin{split}
\tau_d:\dset\rightleftarrows\oper:\nerf_d
\end{split}
\end{equation}
such that the functors $\nerf$ and $\nerf_d$, called the nerve functors,
are fully faithful. Moreover, for any simplicial set $X$, one has
\begin{equation}
\tau_d \, i_!(X)\simeq j_! \, \tau(X)\, .
\end{equation}
By transposition, for any operad $P$, one has
\begin{equation}\label{eq:dendnervevessimplicialnerve}
\nerf\, j^*(P)\simeq i^* \, \nerf_d(P)\, .
\end{equation}
The essential image of the dendroidal nerve functor $\nerf_d$
can be characterized precisely: a dendroidal set $X$ is isomorphic to the
dendroidal nerve of an operad if and only if, for any tree $T$ with a given inner
edge $e$, the map
$$X_T=\Hom_\dset(\Omega[T],X)\to\Hom_\dset(\Lambda^e[T],X)$$
is bijective. In particular, the functor $\tau_d$ sends inner anodyne extensions
to isomorphisms; see \cite[Proposition 5.3 and Theorem 6.1]{dend2}.
\end{paragr}

\begin{paragr}
We call \emph{normal monomorphisms} (\emph{inner anodyne extensions}, respectively)
the elements of the smallest saturated class of maps in $\dset$ which contains the boundary inclusions
(which contains the inner horn inclusions, respectively). One checks that a morphism of
dendroidal sets $X\to Y$ is a normal monomorphism if and only if it is a monomorphism
such that, for any tree $T$ in $\trees$, the action of the group $\mathrm{Aut}(T)$ on the
set $Y(T)-X(T)$ is free; see \cite[Proposition 1.5]{dend3}.
A dendroidal set $X$ is called \emph{normal} if $\varnothing\to X$
is a normal monomorphism. Note that any monomorphism between normal
dendroidal sets is a normal monomorphism, and that, for any map $X\to Y$, if
$Y$ is a normal dendroidal set, then so is $X$. For instance, for any tree $T$, the
dendroidal set $\Omega[T]$ is normal, but the terminal dendroidal set is not.

A morphism of dendroidal sets is called an {inner Kan fibration} if it has the right lifting property with
respect to the class of inner anodyne extensions. An \emph{$\infty$-operad} is
a dendroidal set $X$ such that the map from $X$ to the terminal dendroidal set is
an inner Kan fibration (such objects are also called \emph{inner Kan complexes}
in \cite{dend1,dend2}).

A morphism of dendroidal sets $p:X\to Y$ is called an \emph{isofibration}
if it is an inner Kan fibration and if the morphism of operads $\tau_d(X)\to\tau_d(Y)$
is an isofibration (\ref{def:folkcmcoperads}).
\end{paragr}

\begin{thm}[{\cite[Theorem 2.4]{dend3}}]\label{thm:cmcdendsets}
The category of dendroidal sets is endowed with a unique left proper cofibrantly
generated model category structure such that the
cofibrations are precisely the normal monomorphisms, and
the fibrant objects are precisely the $\infty$-operads.
A morphism between $\infty$-operads is a fibration of this model
category structure if and only if it is an isofibration.

Moreover, the class of weak equivalences is the smallest class of maps $W$
satisfying the following properties:
\begin{itemize}
\item[(a)] in any commutative triangle of $\dset$, if two
maps are in $W$, then so is the third;
\item[(b)] any trivial fibration between $\infty$-operad belongs to $W$;
\item[(c)] any inner anodyne extension is in $W$.
\end{itemize}
\end{thm}

The model categories of operads (\ref{thm:folkcmcoperads}) and
the model category of dendroidal sets (\ref{thm:cmcdendsets}) are related
as follows.

\begin{prop}[{\cite[Proposition 2.5]{dend3}}]\label{prop:adjquillendsetoper}
The adjunction \eqref{eq:nerfdend} is a Quillen pair.
Moreover, the dendroidal nerve functor $\nerf_d$ detects and preserves
weak equivalences, while its left adjoint $\tau_d$ preserves
weak equivalences.
\end{prop}

\begin{rem}
Under the identification $\sset=\dset/\eta$, the model category of Theorem \ref{thm:cmcdendsets}
induces a model category structure on the category of simplicial sets, whose cofibrations
are the monomorphisms, and whose fibrant objects are the \emph{$\infty$-categories} (also known as
\emph{quasi-categories}). In other words, the model category of Theorem \ref{thm:cmcdendsets} contains the Joyal
model category for $\infty$-categories.
\end{rem}

\begin{paragr}\label{def:dendroidaltensorproduct}
There is a closed symmetric monoidal structure on the category
of dendroidal sets, with $\eta$ as the unit object, while we denote by
$X\otimes Y$ and $\sHom(X,Y)$ the corresponding tensor product and
internal Hom functors, respectively.
This monoidal structure is essentially characterized by the fact that the tensor product
functor preserves small colimits in each variable, and by the fact that, for any
trees $S$ and $T$, we have a natural isomorphism
\begin{equation}
\Omega[S]\otimes\Omega[T]\simeq \nerf_d(S\otimes_{\mathit{BV}} T)\, ,
\end{equation}
where $\otimes_{\mathit{BV}}$ denotes the Boardman-Vogt tensor product of
operads; see \cite[Section 5.1]{dend1}. For any two simplicial sets $X$ and $Y$, we have
a natural isomorphism
\begin{equation}
i_!(X)\otimes i_!(Y)\simeq i_!(X)\times i_!(X)=i_!(X\times Y)
\end{equation}
which turns the functor $i_!:\sset\to\dset$ into a symmetric monoidal functor
(if we consider the category of simplicial sets as a symmetric monoidal category, with
the cartesian product as tensor product).
\end{paragr}

\begin{prop}[{\cite[Proposition 2.6]{dend3}}]
The category of dendroidal sets is a symmetric monoidal model category.
\end{prop}

\begin{paragr}\label{def:Homininftyoperads}
If $X$ is a dendroidal set, we will refer to the elements of
the set $X_0=\Hom_\dset(\eta,X)$ as \emph{objects of $X$}.

Let $X$ be an $\infty$-operad.
For an $(n+1)$-tuple of objects $(x_1,\ldots,x_n,x)$ in $X$, $n\geq 0$,
the space of maps $X(x_1,\ldots,x_n;x)$ is obtained
by the pullback below, in which the map $p$ is
the map induced by the inclusion $\partial\Omega[C_n]=\eta\amalg\cdots\amalg\eta\To\Omega[C_n]$
(with $n+1$ copies of $\eta$, corresponding to the $n+1$ objects of the corolla $C_n$).
\begin{equation}\label{eq:def:Homininftyoperads}\begin{split}\xymatrix{
X(x_1,\ldots,x_n;x)\ar[r]\ar[d]&\sHom(\Omega[C_n],X)\ar[d]^p\\
\eta\ar[r]_(.45){(x_1,\ldots,x_n,x)}&X^{n+1}
}\end{split}\end{equation}
Under the equivalence $\sset=\dset/\eta$, we consider
$X(x_1,\ldots,x_n;x)$ as a simplicial set.
Observe that $X(x_1,\ldots,x_n;x)$ is actually a Kan complex; see \cite[Proposition 6.13]{dend3}.
This space of operations may be described with a more homotopy theoretic flavour as follows.
\end{paragr}

\begin{prop}\label{prop:charspacesoperationsMap}
Let $X$ be an $\infty$-operad, and $(x_1,\ldots,x_n,x)$ an $(n+1)$-tuple of
objects in $X$, $n\geq 0$. Then there is a canonical homotopy pullback square
$$\xymatrix{
X(x_1,\ldots,x_n;x)\ar[r]\ar[d]&\Map(\Omega[C_n],X)\ar[d]\\
\Delta[0]\ar[r]_(.38){(x_1,\ldots,x_n,x)}&\Map(\eta,X)^{n+1}
}$$
in the category of Kan complexes.
\end{prop}

\begin{proof}
Using the identification $\sset=\dset/\eta$, one sees that
the pullback square \eqref{eq:def:Homininftyoperads} can be seen as the
composition of the two following pullback squares.
$$\xymatrix{
X(x_1,\ldots,x_n;x)\ar[r]\ar[d]&
i_!i^*\sHom(\Omega[C_n],X)\ar[r]\ar[d]&\sHom(\Omega[C_n],X)\ar[d]^p\\
\eta\ar[r]_(.35){(x_1,\ldots,x_n,x)}&
i_!i^*\sHom(\eta,X)\ar[r]& X^{n+1}
}$$
On the other hand, \cite[Corollary 6.8]{dend3} means that the left hand square
above is itself the composition of two pullback squares
$$\xymatrix{
X(x_1,\ldots,x_n;x)\ar[r]\ar[d]&i_!k(i^*\sHom(\Omega[C_n],X))\ar[r]\ar[d]&
i_!i^*\sHom(\Omega[C_n],X)\ar[d]\\
\eta\ar[r]_(.35){(x_1,\ldots,x_n,x)}&i_!k(i^*\sHom(\eta,X))\ar[r]&
i_!i^*\sHom(\eta,X)
}$$
in which, for an $\infty$-category $C$, $k(C)$ is the maximal Kan subcomplex of $C$
in the sense of Joyal \cite[Corollary 1.5]{joyal}, while \cite[Proposition 6.7]{dend3}
asserts that the middle vertical map is a Kan
fibration. By virtue of \cite[Proposition 3.3]{dend6}, we know that, for any normal
dendroidal set $A$ and any $\infty$-operad $X$, we have a canonical identification
$\Map(A,X)=k(i^*\sHom(A,X))$, whence this proposition.
\end{proof}

\begin{definition}\label{def:equivinftyoperd}
Let $f:X\To Y$ be a morphism of $\infty$-operads.

The map $f$ is \emph{fully faithful} if, for any
$(n+1)$-tuple of objects $(x_1,\ldots,x_n,x)$ in $X$,
the morphism
$$X(x_1,\ldots,x_n;x)\To Y(f(x_1),\ldots,f(x_n);f(x))$$
is a simplicial homotopy equivalence.

The map $f$ is \emph{essentially surjective} if
the morphism of operads $\tau_d(X)\To \tau_d(Y)$ is essentially surjective.
\end{definition}

\begin{rem}
For an $\infty$-operad $X$, the set of isomorphism classes of objects in the
category $j^*\, \tau_d(X)\simeq \tau\, i^* (X)$ is canonically in bijection with the
set of connected components of the Kan complex $\Map(\eta,X)$;
see \cite[4.1]{dend3}. Therefore, using Proposition \ref{prop:charspacesoperationsMap},
we see that the notions of essentially surjective map and of fully faithful map
have a purely homotopic meaning in the model category of dendroidal sets.
In fact, we have:
\end{rem}

\begin{thm}[{\cite[Theorem 3.11]{dend6}}]\label{thm:charwedsetfullfaithesssurj}
A morphism between $\infty$-operads is a weak equivalence if and only if it is
fully faithful and essentially surjective.
\end{thm}

\section{Adding free cells}\label{sect:3}

This section is quite technical: we will study some specific pushouts
in the model category of dendroidal sets. This will be used later to produce a Quillen
equivalence between the homotopy theories of Segal operads and of simplicial
operads (see the proof of Lemma \ref{lemma:pushoutnervepreoper}).
Our calculations will also have implications for the computation of
homotopy pushouts in the model category of simplicial operads
(see Theorem \ref{thm:sigmacofleftproper} and its corollaries).

\begin{paragr}\label{def:addfreecell}
Let $m\geq 0$ be a fixed integer, and write $f$ for
the operation corresponding to the unique vertex of the corolla $C_m$.
We find it convenient to write $\partial C_m$
for the operad $\eta\amalg\ldots\amalg\eta$ ($m+1$ copies of $\eta$),
so that we have a morphism $\partial C_m\to C_m$ which is bijective on objects.
Given an operad $P$ (in the category of sets) as well as an $(m+1)$-tuple of objects
$(x_0,\ldots,x_m,x)$ in $P$, we define the operad $P[f]$ as the pushout below.
\begin{equation}\label{eq:addfreecell1}
\begin{split}
\xymatrix{
\partial C_m\ar[rr]^{(x_0,\ldots,x_m,x)}\ar[d]&&P\ar[d]\\
C_m\ar[rr]&&P[f]
}
\end{split}
\end{equation}
We also define the dendroidal set $\nerf_d(P)[f]$ as the following pushout.
\begin{equation}\label{eq:addfreecell2}
\begin{split}
\xymatrix{
\partial \Omega[C_m]\ar[rr]^{(x_0,\ldots,x_m,x)}\ar[d]&&\nerf_d(P)\ar[d]\\
\Omega[C_m]\ar[rr]&&\nerf_d(P)[f]
}
\end{split}
\end{equation}
We have a canonical comparison map
\begin{equation}\label{compmap:addfreecell}
\nerf_d(P)[f]\to\nerf_d(P[f])
\end{equation}
which happens to be bijective on objects.
\end{paragr}

\begin{prop}\label{prop:descriptionofPf}
If the operad $P$ is $\Sigma$-cofibrant, then the morphism \eqref{compmap:addfreecell}
is a normal monomorphism between normal dendroidal sets.
\end{prop}

\begin{proof}
As the map $P\to P[f]$ is a $\Sigma$-cofibration between $\Sigma$-cofibrant operads,
it is a monomorphism\footnote{The proof of this fact can be found in
\cite[Proposition 5.1]{BMOp} in the case of an operad with one object. The case with several objects
is proved similarly; see also the Appendix of \cite{BMOp2}, where a little mistake
of \cite{BMOp} is corrected.}.
Therefore, as the dendroidal nerve of a $\Sigma$-cofibrant operad is normal,
the dendroidal nerve of this map is a normal monomorphism
between normal dendroidal sets. It is thus sufficient to check that
the morphism \eqref{compmap:addfreecell} is a monomorphism.
For this purpose, we just have to see that the commutative square
$$\xymatrix{
\partial \Omega[C_m]\ar[rr]^{(x_0,\ldots,x_m,x)}\ar[d]&&\nerf_d(P)\ar[d]\\
\Omega[C_m]\ar[rr]&&\nerf_d(P[f])
}$$
is cartesian. As the dendroidal nerve functor preserves pullbacks, this follows from the fact
that the commutative square of operads \eqref{eq:addfreecell1} is cartesian.
\end{proof}

\begin{paragr}
The main result of this section is that the map \eqref{compmap:addfreecell}
is an inner anodyne extension (\ref{prop:addfreecellweakequiv}). This will require further preparation.

Given a tree $T$, an element of $\nerf_d(P[f])_T=\Hom(T,P[f])$ is completely determined by a
labelling of the edges of $T$ by objects of $P$ together with a (compatible) labelling of the
vertices of $T$ by operations of $P[f]$ (here, we are just reformulating that $\nerf_d(P[f])$ is
a strict inner Kan complex, and thus satisfies the strict dendroidal Segal condition; see
\cite[Corollary 2.7]{dend6}). Now, an operation in $P[f]$ may be represented by
another tree $S$ whose vertices are labelled by operations of $P$ or by the letter $f$; such a representation
will be said \emph{canonical} if no inner edge of $S$ has both vertices labelled by elements of $P$
(otherwise we may compose them) and if no unary vertex of $S$ is labelled by an identity of $P[f]$.
If $q$ is an operation in $P[f]$, represented by a canonical
labelling of a tree $S$, we will write $m(q)$ for the \emph{multiplicity
of $f$ in $q$}, that is the number of vertices of $S$ which are labelled by the freely added operation $f$.
Note that the number $m(q)$ does not depend on the choice of a canonical presentation of $q$.
For a tree $T$ and an element $q$ of $\nerf_d(P[f])_T$, we define $m(q)=\sum_{v\in V(T)}m(q_v)$,
where $V(T)$ denotes the set of vertices of $T$, while, for each $v\in V(T)$, the letter $q_v$
stands for the operation of $P[f]$ attached to $v$.

For an integer $k\geq 0$ and for any tree $T$, we define the set $\nerf_d(P[f])^{(k)}_T$ as
\begin{equation}\label{eq:deffiltrationaddfreecell}
\nerf_d(P[f])^{(k)}_T=\{q\in\nerf_d(P[f])_T\, | \, m(q)\leq k\}\, .
\end{equation}
This defines a dendroidal subcomplex $\nerf_d(P[f])^{(k)}\subset\nerf_d(P[f])$.
Therefore, the dendroidal set $\nerf_d(P[f])$ admits an increasing filtration
\begin{equation}\label{eq:filteraddfreecell}
\nerf_d(P)=\nerf_d(P[f])^{(0)}\subset \nerf_d(P[f])^{(1)}\subset\ldots\subset\nerf_d(P[f])^{(k)}\subset\ldots
\end{equation}
such that $\nerf_d(P[f])=\bigcup_{k\geq 0}\nerf_d(P[f])^{(k)}$.
The filtration above can be refined as
\begin{equation}\label{eq:filteraddfreecellrefined}
\nerf_d(P[f])^{(k)}=A^{(k)}_0\subset A^{(k)}_1\subset\ldots A^{(k)}_n\subset\ldots \subset \nerf_d(P[f])^{(k+1)}
\end{equation}
in which, for each integer $n\geq 0$, $A^{(k)}_n$ is the dendroidal subcomplex of $\nerf_d(P[f])^{(k+1)}$
defined as follows:
\begin{equation}\label{eq:deffilteraddfreecellrefined}
A^{(k)}_n=\nerf_d(P[f])^{(k)}\cup \bigcup_{q\in U^{(k)}_n} \text{image of $q$}\, ,
\end{equation}
where $U^{(k)}_n$ is the set of all non-degenerate dendrices $q:\Omega[T]\to\nerf_d(P[f])^{(k+1)}$ such that
$T$ has at most $n+k$ vertices and such that, for any vertex $v$ of $T$, the corresponding
decoration $q_v$ satisfies $m(q_v)=0$ or is a corolla decorated by an isolated occurrence of $f$.
Remark that, if $T$ is a tree with at least one vertex,
any non-degenerate element $q$ in $\nerf_d(P[f])_T$ may be obtained as the inner face (along $T\to T'$)
of an element of $\nerf_d(P[f])_{T'}$ which belongs to $U_{n'}$ for a big enough integer $n'$.
In other words, we have that \smash{$\nerf_d(P[f])^{(k+1)}=\bigcup_{n\geq 0}A^{(k)}_n$}.
Note that we have an identification
\begin{equation}\label{eq:notefilteraddfreecellrefined}
A^{(0)}_1=\nerf_d(P)[f]\, .
\end{equation}
\end{paragr}

\begin{lemma}\label{lemma:suffcondinnanod}
Let $X\to Y$ be a normal monomorphism of dendroidal sets, and $n\geq 1$ an integer.
Assume that any non-degenerate dendrex $s:\Omega[T]\to Y$ which does not factor
through $X$ has the following properties.
\begin{itemize}
\item[(i)] The tree $T$ has exactly $n+1$ vertices.
\item[(ii)] The map $\Omega[T]\times_Y X\to \Omega[T]$ is an inner anodyne
extension.
\end{itemize}
Then the map $X\to Y$ is inner anodyne.
\end{lemma}

\begin{proof}
There exists a set $\Sigma$ of non-degenerate dendrices $s:\Omega[T]\to Y$
which do not factor through $X$, the trees $T$ having exactly
$n+1$ vertices, such that the commutative square
$$\xymatrix{
\coprod_{s\in\Sigma}\partial\Omega[T]\ar[r]\ar[d]&X\cup\mathit{Sk}_{n}(Y)\ar[d]\\
\coprod_{s\in\Sigma}\Omega[T]\ar[r]&Y
}$$
is cocartesian (apply \cite[Lemme 8.1.34]{Ci3} for $A=\trees$).
As the left vertical map of this square is a monomorphism, this square is also
cartesian. This implies that the cartesian square
$$\xymatrix{
\coprod_{s\in\Sigma}\Omega[T]\times_X Y\ar[r]\ar[d]&X\ar[d]\\
\coprod_{s\in\Sigma}\Omega[T]\ar[r]&Y
}$$
is cocartesian, and ends the proof of the lemma, for the class
of inner anodyne extensions is closed under small sums and pushouts.
\end{proof}

\begin{lemma}\label{lemma:inneranodynefreecell}
If $P$ is $\Sigma$-cofibrant, then, for any integers $n,k\geq 0$
such that $n+k$ is positive, the inclusion
$A^{(k)}_n\subset A^{(k)}_{n+1}$ is an inner anodyne extension.
\end{lemma}

\begin{proof}
We already know that this map is a normal monomorphism (Proposition \ref{prop:descriptionofPf}).
In order to simplify the notations, we will write \smash{$A^{(k)}_n=A_n$}
and, similarly, \smash{$U^{(k)}_n=U_n$}. Henceforth, we assume that $n+k>0$.
Let $q:\Omega[T]\to A_{n+1}$ be a non-degenerate dendrex which does not factor through $A_n$.
This means that $T$ must have exactly $n+1+k$ vertices and that $q\in U_{n+1}$.
Let $I_q$ be the set of inner edges of $T$ which are adjacent to at least one of the vertices
labelled by $f$, and write $\Lambda^{I_q}[T]$ for the union in $\Omega[T]$ of the images
of all the outer face maps as well as of the inner face maps
$\partial_e:\Omega[T/e]\to\Omega[T]$, for $e\notin I_q$.
By virtue of \cite[Lemma 5.1]{dend2}, the inclusion of $\Lambda^{I_q}[T]$ into $\Omega[T]$
is an inner anodyne extension (the set $I_q$ is not empty because we assumed $n+k$ to be positive).
Therefore, by Lemma \ref{lemma:suffcondinnanod}, it is sufficient to prove that we have a cartesian
square of the form below.
$$\xymatrix{
\Lambda^{I_q}[T]\ar[r]\ar[d]&A_n\ar[d]\\
\Omega[T]\ar[r]^q&A_{n+1}
}$$
As any map of $\trees$ factors as a (split) epimorphism followed by a monomorphism, it is
sufficient to prove that, for any injective map $\partial:\Omega[S]\to\Omega[T]$,
the composed morphism
$$\Omega[S]\xrightarrow{ \ \partial \ }\Omega[T]\xrightarrow{ \ q \ } A_{n+1}$$
factors through $A_n$ if and only if $\partial$ factors through $\Lambda^{I_q}[T]$.

Assume first that $\partial :S\to T$ is an outer face.
We then have that $m(\partial(q))\leq k+1$ and that $S$ has at most $n+k$ vertices.
In fact, one easily sees that, if $q\partial$ does not belong to $\nerf_d(P[f])^{(k)}$,
then it must be in $U_n$ (because an outer face does not affect compositions).
To finish the proof of the lemma, it is now sufficient
to prove that, for any injective map $\partial:\Omega[S]\to\Omega[T]$ which does not
factor through an outer face of $T$, the composed morphism $q\partial$ factors through $A_n$
if and only if $\partial$ factors through $\Lambda^{I_q}[T]$.
The assumption that $\partial$ does not factor through an outer face of $T$ means that
$\partial$ is obtained as a finite composition of inner face maps, i.e. is a map obtained by
contracting a finite family $J$ of inner edges of $T$. The number of vertices of $S$ is then
$n+1+k-j$, where $j$ is the number of elements in the set $J$.
We also have that $m(q\partial)=k+1$ (because going from $T$ to $S$ via inner faces
only involves composition of some operations in $P$,
so that the global number of occurrences of $f$ doesn't change).
Therefore, the map $q\partial$ cannot factor through $\nerf_d(P[f])^{(k)}$.
This implies that $q\partial$ factors through $A_n$ if and only if $q\partial$ is
a face of an element of $U_n$.
This latter property may be reformulated as the
existence of an inner edge $e\notin I_q$, such that $q\partial$
factors through the composition
of $q$ with the inner face $\partial/e:\Omega[T/e]\To\Omega[T]$ (so that
$q(\partial/e)$ still has $k+1$ isolated occurrences of $f$ and only $n+k$ vertices). 
The dendrex $q$ being non-degenerate, $q\partial$ factors through
$q(\partial/e)$ if and only if $\partial$ factors through $\partial/e$, which
precisely means that $e\in J$.
This shows that $\Omega[T]\times_{A_{n+1}}A_n\simeq \Lambda^{I_q}[T]$, and
completes the proof.
\end{proof}

\begin{prop}\label{prop:addfreecellweakequiv}
If the operad $P$ is $\Sigma$-cofibrant, then the morphism \eqref{compmap:addfreecell}
is an inner anodyne extension.
\end{prop}

\begin{proof}
The class of inner anodyne extensions is closed under countable composition.
Therefore, this proposition follows straight away from Lemma \ref{lemma:inneranodynefreecell} and from
the identification \eqref{eq:notefilteraddfreecellrefined}.
\end{proof}

\begin{paragr}\label{paragr:addeverelfreeells}
Let $P$ be a $\Sigma$-cofibrant operad, $(n_i)_{i\in I}$ a small family of integers, and
$u_i:\partial C_{n_i}\to P$, $i\in I$ a family of maps. We get the following comparison map of pushouts
\begin{equation}\label{eq:addeverelfreeells}
\big(\coprod_{i\in I}\Omega[C_{n_i}]\big)\amalg_{\big(\coprod_{i\in I}\partial\Omega[C_{n_i}]\big)}\nerf_d(P)
\to \nerf_d\Big(\big(\coprod_{i\in I}C_{n_i}\big)\amalg_{\big(\coprod_{i\in I}\partial C_{n_i}\big)} P\Big)\, .
\end{equation}
\end{paragr}

\begin{cor}\label{cor:addeverelfreeells}
The map \eqref{eq:addeverelfreeells} is an inner anodyne extension.
\end{cor}

\begin{proof}
Given a subset $J$ of $I$, write
$$A_J=\big(\coprod_{i\in J}\Omega[C_{n_i}]\big)\amalg_{\big(\coprod_{i\in J}\partial\Omega[C_{n_i}]\big)}\nerf_d(P)
\quad\text{and}\quad
B_J=(\big(\coprod_{i\in J}C_{n_i}\big)\amalg_{\big(\coprod_{i\in J}\partial C_{n_i}\big)} P\, .$$
We proceed by transfinite induction to prove that the comparison map $A_J\to\nerf_d(B_J)$
is an anodyne extension for all subsets $J$ of $I$.
For this purpose, we notice that, for a subset $J\subset I$ and an element $j\in J$,
if $f_j$ denotes the unique non trivial operation of $C_{n_j}$,
we have the following pushout.
$$\xymatrix{A_{J-\{j\}}\ar[r]\ar[d]& \nerf_d(B_{J-\{j\}})\ar[d]\\
A_J\ar[r]&\nerf_d(B_{J-\{j\}})[f_j]}$$
Therefore, by transfinite induction, the map $A_J\to \nerf_d(B_{J-\{j\}})[f_j]$
is an inner anodyne extension. On the other hand, by virtue of
Proposition \ref{prop:addfreecellweakequiv}, the map
$$\nerf_d(B_{J-\{j\}})[f_j]\to \nerf_d(B_{J-\{j\}}[f_j])=\nerf_d(B_J)$$
is an inner anodyne extension, which ends the proof of this corollary.
\end{proof}

\section{The Boardman-Vogt construction}\label{sect:4}

\begin{paragr}\label{paragr:recallBMtrees}
Recall from \cite{BV,BMOp0,BMOp1} the Boardman-Vogt construction
$$\bvoper(\mathcal P)\to\mathcal P$$
of a simplicial operad $\mathcal P$ (with respect to the interval ${\Delta[1]}$).
The map $\bvoper(\mathcal P)\to\mathcal P$ is a weak equivalence, while the $\bvoper(\mathcal P)$-algebras
are the ``$\mathcal P$-algebras up to homotopy'' in the usual Quillen model category of simplicial sets.
We will only need to understand the Boardman-Vogt construction for the operads freely
generated by trees; see \cite[pages 819--820]{BMOp0} and \cite[Remarks 7.2 and 7.3]{dend2}
for a very precise explicit description in this case. We recall the main properties of this construction below.

Let $T$ be a tree, seen as an operad. By virtue of \cite[Theorem 3.5]{BMOp1},
the \emph{Boardman-Vogt resolution of $T$}
gives a morphism of operads $\varepsilon:\bvoper(T)\to T$ which is functorial in $T$ (seen as an object of the category $\trees$)
and such that:
\begin{itemize}
\item[(a)] the simplicial operad $\bvoper(T)$ is cofibrant in the model category structure
of Theorem \ref{thm:cmcsoper};
\item[(b)] the map $\varepsilon: \bvoper(T)\to T$ is bijective on objects and
a weak equivalence (in particular, the map $\varepsilon$
induces a natural isomorphism of operads $\pi_0(\bvoper(T))\simeq T$).
\end{itemize}
\end{paragr}

\begin{paragr}\label{paragr:explicitBVtrees}
As this will be needed later, we now recall the explicit description of the Board\-man-Vogt construction
for trees.

Let $T$ be a tree, seen as an operad. It admits the following description.
Its objects are the edges of $T$. For an $(n+1)$-tuple of edges $(e_1,\ldots,e_n,e)$, $n\geq 0$,
the set of operations $T(e_1,\ldots,e_n;e)$ is the set of subtrees $V$ of $T$
such that $e_1,\ldots,e_n$ are exactly the input edges of $V$, while $e$ is the root of $V$
(whence there is at most one such operation).
The simplicial set of operations from $(e_1,\ldots,e_n)$ to $e$ in $\bvoper(T)$ is
\begin{equation}\label{eq:defBVtrees}
\bvoper(T)(e_1,\ldots,e_n;e)=\coprod_{V\in T(e_1,\ldots,e_n;e)}{\Delta[1]}^{i(V)}
\end{equation}
where, for a subtree $V\subset T$, $\mathit{i}(V)$ denotes the set of inner edges of $V$
(in particular, this Kan complex is either empty, or isomorphic to the power ${\Delta[1]}^{i(V)}$).
The composition in $\bvoper(T)$ is defined by grafting trees, where the newly arising internal edges
are assigned length $1$; we refer the reader to \cite[Remark 7.3]{dend2} for a precise description
of the composition law.

We will now study the functoriality of this construction (we will consider successively
the case of a bijection, of an inner face, of an outer face, and of a degeneracy).

Let $\varphi:S\to T$ be a morphism of trees (remark that, once a planar structure is chosen for $T$, there is
a unique planar structure on $S$ such that $\varphi$ preserves the planar structures).
Consider an $(n+1)$-tuple of edges $(d_1,\ldots,d_n,d)$ of $S$, and write $(e_1,\ldots,e_n,e)$
for its image by $\varphi$. We want to understand the morphism of simplicial sets
\begin{equation}\label{eq:deffunctBVtrees}
\bvoper(\varphi):\bvoper(S)(d_1,\ldots,d_n;d)\to \bvoper(T)(e_1,\ldots,e_n;e)\, .
\end{equation}
If $S(d_1,\ldots,d_n;d)=\varnothing$, there is nothing to say.
From now on, we will denote by $U$ an element of $S(d_1,\ldots,d_n;d)$
(which we presume to exist),
and write $V$ for the element of $T(e_1,\ldots,e_n;e)$.

If $\varphi$ is a bijection, then it maps $U$ isomorphically onto $V$, and the
map \eqref{eq:deffunctBVtrees} is the isomorphism associated to the
induced bijection $i(U)\simeq i(V)$.

If $\varphi$ is an inner face, say $\varphi=\partial_t:T/t\to T$, for $t$ an inner edge of $T$,
then there are two cases. Either $t\notin V$, and then $U=V$, so that the map
\eqref{eq:deffunctBVtrees} is the identity, or $t\in V$, and then
$U$ is obtained from $V$ by contracting the edge $t$.
In the latter case, the map \eqref{eq:deffunctBVtrees}
is obtained as the product of the identity of ${\Delta[1]}^{i(U)}$ with the inclusion
$\{0\}\to {\Delta[1]}$, composed with the canonical isomorphism
${\Delta[1]}^{i(U)}\times {\Delta[1]}\simeq {\Delta[1]}^{i(V)}$
provided by the identification $i(U)=i(V)-\{t\}$.

If $\varphi$ is an outer face obtained by removing a vertex $v$ in $T$
(necessarily with the property that all the edges incident to $v$ are outer),
then $v$ is not in $V$, and we must have $U=V$ (because
we already assume that the leaves $e_1,\ldots,e_n$ and root $e$
of $V$ belong to $S$ as $d_1,\ldots,d_n$ and $d$, respectively),
so that the map \eqref{eq:deffunctBVtrees} is the identity.

If $\varphi$ is a degeneracy, say $\varphi=\sigma_s$, for a unary vertex $s$ in $S$,
then there are again two cases. Either $s$ is not in $U$, and then $\varphi$ maps $U$
isomorphically onto $V$, so that  the map \eqref{eq:deffunctBVtrees} is the isomorphism
associated to the induced bijection $i(U)\simeq i(V)$, or $s$ occurs in $U$.
In the latter case, $V$ is obtained from $U$ by identifying the two edges, say $a_1$ and $a_2$,
on either side of $s$. Let us call $a$ the image of $a_1$ and $a_2$ in $V$.
The map \eqref{eq:deffunctBVtrees} is the morphism
$${\Delta[1]}^{i(U)}\simeq {\Delta[1]}^{i(U)-\{a_1,a_2\}}\times {\Delta[1]}^2
\to  {\Delta[1]}^{i(V)-\{a\}}\times {\Delta[1]}\simeq  {\Delta[1]}^{i(V)}$$
obtained as the product of the isomorphism induced by the bijection
$i(U)-\{a_1,a_2\}\simeq i(V)-\{a\}$ with the map
${\Delta[1]}^2\to{\Delta[1]}$ given by the formula $(x,y)\longmapsto \mathrm{max}\{x,y\}$.
\end{paragr}

\begin{paragr}\label{paragr:defWdendsets}
We denote by
\begin{equation}\label{eq:defWdendsets}
\bv:\dset\to\soper
\end{equation}
the colimit preserving functor whose 
composition with the Yoneda embedding $\trees\to\dset$
coincides with the Boardman-Vogt construction on trees; i.e., for any tree $T$, one has
$$\bv(\Omega[T])=\bvoper(T)\, .$$
This functor can also be described as the left adjoint to the homotopy coherent nerve functor
\begin{equation}\label{eq:defhnerve}
\hnerf=W^*:\soper\to\dset
\end{equation}
defined by
$$\hnerf(\mathcal P)_T=\Hom_\soper(\bvoper(T),\mathcal P)$$
for any simplicial operad $\mathcal P$ and any tree $T$.

The naturality of the map $\varepsilon:\bvoper(T)\to T=\pi_0(\bvoper(T))$
now extends to a natural morphism
of simplicial operads
\begin{equation}\label{eq:BVtruncation}
\bv(X)\to\tau_d(X)\, .
\end{equation}
\end{paragr}

\begin{prop}\label{prop:truncateBV}
For any dendroidal set $X$, the natural map $\pi_0\bv(X)\to\tau_d(X)$
is an isomorphism of operads.
\end{prop}

\begin{proof}
As both functors $\pi_0\bv$ and $\tau_d$ commute with small colimits,
it is sufficient to check this when $X$ is of the form $\Omega[T]$ for a tree $T$, in which
case this follows from property (b) of \ref{paragr:recallBMtrees}.
\end{proof}

\begin{prop}\label{prop:BVpreservescofandinneranod}
The functor $\bv:\dset\to\soper$ sends normal monomorphisms to
cofibrations and inner anodyne extensions to trivial cofibrations.
\end{prop}

\begin{proof}
In order to prove that the functor $\bv$ sends normal monomorphisms to
cofibrations, it is sufficient to check that it sends the generating normal
monomorphisms to cofibrations. In other words, it is sufficient to check that,
for any tree $T$, the map
$$i:\bv(\partial\Omega[T])\to\bv(\Omega[T])$$
is a cofibration. If $T=\eta$, then the map
$\bv(\partial\Omega[\eta])=\varnothing\to\bv(\Omega[\eta])=\eta$
is known to be a cofibration. Therefore, we may assume that $T$ has at least one
vertex. In this case, the map $i$ is bijective on objects.
For an $(n+1)$-tuple of edges $(e_1,\ldots,e_n,e)$ in $T$, $n\geq 0$,
one checks, by inspection of the definition, that, unless $\{e_1,\ldots e_n\}$
is exactly the set of input edges of $T$ and $e=r$ is the root of $T$, the map
$$\bv(\partial\Omega[T])(e_1,\ldots,e_n;e)\to\bv(\Omega[T])(e_1,\ldots,e_n;e)$$
is the identity. In the case when this is not the identity, this map
can be identified with the inclusion
$$\partial {\Delta[1]}^d\to {\Delta[1]}^d\, ,$$
where $d$ is the number of inner edges of $T$, and where,
for an integer $m\geq 0$, one defines $\partial {\Delta[1]}^m$ by induction
as follows: $\partial {\Delta[1]}^0=\varnothing$, $\partial\Delta[1]=\{0\}\amalg\{1\}$,
and, for $m>1$,
$$\partial {\Delta[1]}^m=
\partial {\Delta[1]}^{m-1}\times {\Delta[1]}\cup {\Delta[1]}^{m-1}\times\partial {\Delta[1]}
\subset {\Delta[1]}^{m-1}\times {\Delta[1]}\simeq {\Delta[1]}^m\, .$$
We deduce that, in the case where $T$ has at least one vertex,
the morphism $i$ has the left lifting property with respect
the class of local trivial fibrations. Therefore, by virtue of
Proposition \ref{prop:trivfibrationsoper}, the morphism $i$ must be a cofibration for any tree $T$.

Similarly, in order to prove that the functor $\bv$ sends inner anodyne
extensions to trivial cofibrations, it is sufficient to prove that, for
any tree $T$ with a given inner edge $t$, the map
$$j:\bv(\Lambda^t[T])\to\bv(\Omega[T])$$
is a trivial cofibration. As before, one sees that this map is always
bijective on objects. For any $(n+1)$-tuple of edges $(e_1,\ldots,e_n,e)$ in $T$, $n\geq 0$,
unless $\{e_1,\ldots e_n\}$ is exactly the set of input edges of $T$ and $e=r$ is the root of $T$, the map
$$\bv(\Lambda^t[T])(e_1,\ldots,e_n;e)\to\bv(\Omega[T])(e_1,\ldots,e_n;e)$$
is the identity. Otherwise it can be identified canonically with the trivial cofibration
of simplicial sets:
$${\Delta[1]}^{i(T)-\{t\}}\times \{1\}\cup \partial {\Delta[1]}^{i(T)-\{t\}}\times {\Delta[1]}\to {\Delta[1]}^{i(T)}\, .$$
This implies that the morphism $j$ has the left lifting property with respect to
the class of local fibrations. Proposition \ref{prop:fibrationsoper} thus implies that $j$
must be a trivial cofibration.
\end{proof}

\begin{cor}\label{cor:hcnervefibrant}
For any fibrant simplicial operad $\mathcal P$, its homotopy coherent nerve
$\hnerf(\mathcal P)$ is an $\infty$-operad.
\end{cor}

\begin{rem}
The preceding corollary is \cite[Theorem 7.1]{dend2} in the case where the underlying
model category is the one of simplicial sets. Note that the proof
of Proposition \ref{prop:BVpreservescofandinneranod} is just a slightly more precise
version of the proof of \emph{loc. cit}. By this, we mean that the analog
of Proposition \ref{prop:BVpreservescofandinneranod} is true
(with essentially the same proof) if we replace
the category of simplicial operads by the category of operads in an
adequate symmetric monoidal model category.
\end{rem}

\begin{prop}\label{prop:truncatehcnerve}
For any fibrant simplicial operad $\mathcal P$, there is
a canonical isomorphism $\tau_d(\hnerf(\mathcal P))\simeq\pi_0(\mathcal P)$.
\end{prop}

\begin{proof}
It follows from property (b) of \ref{paragr:recallBMtrees} that, for any operad $P$
in the category of sets, there is a natural isomorphism
$$\nerf_d(P)\simeq \hnerf(P)\, .$$
indeed, for any tree $T$, one has
$$\Hom_\oper(T,P)=\Hom_\oper(\pi_0(\bvoper(T)),P)\simeq\Hom_\soper(\bvoper(T),P)\, .$$
If $\mathcal P$ is a simplicial operad, then the map $\mathcal P\to\pi_0(\mathcal P)$
induces a morphism
$$\hnerf(\mathcal P)\to\hnerf(\pi_0(\mathcal P))\simeq\nerf_d(\pi_0(\mathcal P))\, ,$$
from which we obtain, by adjunction, a canonical morphism
$$\tau_d(\hnerf(\mathcal P))\to \pi_0(\mathcal P)\, .$$
Note that the functors $\tau_d$, $\hnerf$ and $\pi_0$ don't affect objects.
Therefore, this morphism is bijective on objects.
Let $(x_1,\ldots,x_n,x)$ an $(n+1)$-tuple of objects of $\mathcal P$, with
$n\geq 0$. We have to prove that, if $\mathcal P$ is fibrant, the induced map
$$\tau_d(\hnerf(\mathcal P))(x_1,\ldots,x_n;x)\to \pi_0(\mathcal P(x_1,\ldots,x_n;x))$$
is bijective. Now, we remark that $\bvoper(T)=T$ whenever $T$ is a tree with no inner edges
(e.g. $T=\eta$ or $T=C_n$).
This implies that the set of $0$-simplices of $\hnerf(\mathcal P)(x_1,\ldots,x_n;x)$ can be identified with
the set of $0$-simplices of the simplicial set $\mathcal P(x_1,\ldots,x_n;x)$.
As $\hnerf(\mathcal P)$ is an $\infty$-operad,
by virtue of \cite[Proposition 6.3, Lemma 6.4 and Proposition 6.10]{dend2},
we know that the set of operations $\tau_d(\hnerf(\mathcal P))(x_1,\ldots,x_n;x)$ is the quotient
of the set of $0$-simplices of $\hnerf(\mathcal P)(x_1,\ldots,x_n;x)$ by the
equivalence relation $\sim$ defined as follows.
Write $C^+_n$ for the tree obtained from the corolla $C_n$ by grafting a new edge under the root.
We have two outer faces $u,v:C_n\to C^+_n$: the face $u$ misses the new root (whence is an
outer face), while the face $v$ misses the old one (and is thus an inner face).
There is also the outer face $d:C_1\to C^+_n$ which corresponds
to the map from the old root to the new one in $C^+_n$.
For two elements
$$p,q\in \hnerf(\mathcal P)(x_1,\ldots,x_n;x)\subset\hnerf(\mathcal P)_{C_n}
\simeq\Hom_\soper(C_n,\mathcal P)$$
we say that $p\sim q$ if there exists a morphism $h:\Omega[C^+_n]\to \hnerf(\mathcal P)$
such that $hu=p$ and $hv=q$, and such that $hd$ is degenerate (i.e. factors through $\eta$).
To finish the proof, it is thus sufficient to prove that, if two $0$-simplices $p$ and $q$ of the Kan complex
$\mathcal P(x_1,\ldots,x_n;x)$ are in the same connected component, then $p\sim q$.
For this, we have to understand the elements $h$ of the set
$$\hnerf(\mathcal P)_{C^+_n}=\Hom_\soper(\bvoper(C^+_n),\mathcal P)$$
such that $hd$ is degenerate.
Let $a_1,\ldots,a_n$ be the input edges of $C^+_n$, and $r$ be the root edge,
while $r'$ is the remaining edge of $C^+_n$.
Let $f:\{a_1,\ldots,a_n,r',r\}\to \ob\mathcal P$ be the map defined by $f(a_i)=x_i$
and $f(r)=f(r')=x$. We then have an identification
$$\bvoper(C^+_n)(a_1,\ldots,a_n;r)={\Delta[1]}\, ,$$
which induces a map
$$\Hom_{\soper_{ \{a_1,\ldots,a_n,r',r\}}}(\bvoper(C^+_n),f^*\mathcal P)\to
\Hom_{\sset}({\Delta[1]},\mathcal P(x_1,\ldots,x_n;x))\, .$$
Let \smash{$\Hom_{\soper_{ \{a_1,\ldots,a_n,r',r\}}}(\bvoper(C^+_n),f^*\mathcal P)_{\mathit{deg}}$}
be the set of maps $\varphi:\bvoper(C^+_n)\to f^*\mathcal P$ in \smash{$\soper_{ \{a_1,\ldots,a_n,r',r\}}$}
which send the unique map $r'\to r$ in $\bvoper(C^+_n)$ to the identity of $x$.
Then the restricted map
$$\Hom_{\soper_{ \{a_1,\ldots,a_n,r',r\}}}(\bvoper(C^+_n),f^*\mathcal P)_{\mathit{deg}}\to
\Hom_{\sset}({\Delta[1]},\mathcal P(x_1,\ldots,x_n;x))\, .$$
is easily seen to be bijective.
If $p$ and $q$ are in the same connected component of the Kan
complex $\mathcal P(x_1,\ldots,x_n;x)$, then we can choose a path
$$H:{\Delta[1]} \to \mathcal P(x_1,\ldots,x_n;x)$$
such that $H(0)=q$ and $H(1)=p$. This morphism extends uniquely
to a morphism of simplicial operads $h:\bvoper(C^+_n) \to \mathcal P$
which sends $a_i$ to $x_i$ and which sends the map $r'\to r$ to the identity of $x$,
such that the map
$$\bvoper(C^+_n)(a_1,\ldots,a_n;r)\to \mathcal P(x_1,\ldots,x_n;x)$$
is $H$. Thus we must have the relations $hu=p$ and $hv=q$ (see \ref{paragr:explicitBVtrees}).
In other words: $p\sim q$.
\end{proof}

\begin{prop}\label{prop:BVleftQuillen}
The functor $\bv:\dset\to\soper$ is a left Quillen functor.
\end{prop}

\begin{proof}
As we already know that $\bv$ preserves cofibrations (Proposition \ref{prop:BVpreservescofandinneranod}),
it is sufficient to check that its right adjoint $\hnerf$ sends fibrations
between fibrant objects of Theorem \ref{thm:cmcreedysegaloperads} to fibrations; see \cite[Proposition 7.15]{JT}.
Let $f:\mathcal P\to \mathcal Q$ be a fibration between fibrant simplicial operads.
It follows from Proposition \ref{prop:BVpreservescofandinneranod} that $\hnerf(f)$
is an inner Kan fibration between $\infty$-operads.
Therefore, the map $\hnerf(f)$ is a fibration if and only if its image by the
functor $\tau_d$ is an isofibration (\ref{thm:cmcdendsets}).
But, by definition of fibrations of simplicial operads, we already know that $\pi_0(f)$ is an isofibration
of operads. Proposition \ref{prop:truncatehcnerve} ends the proof.
\end{proof}

\section{Segal operads: the normal model category structure}\label{sect:5}

This section is a reminder of the homotopy theory of Segal operads and of complete dendroidal Segal spaces,
as developed in our second paper \cite{dend6}.

\begin{paragr}\label{paragr:preliminariessdsets}
The category of simplicial dendroidal sets (i.e. of simplicial objects in $\dset$)
is denoted by $\sdset$. We will view the category $\dset$ of dendroidal sets and the
category $\sset$ of simplicial sets as full subcategories of $\sdset$ in the obvious way
(through the constant presheaf functors). In particular, if we have a simplicial set $K$
and a dendroidal set $A$, the cartesian product $K\times A$ will mean the cartesian
product of $K$ and $A$, seen as simplicial dendroidal sets; in other words, for an integer $n\geq 0$
and a tree $T$, we have
\begin{equation}
(K\times A)_{n,T}=K_n\times A_T\, .
\end{equation}

A map between simplicial dendroidal sets $X\to Y$ is a \emph{normal monomorphism} if, for any integer $n\geq 0$,
the morphism of dendroidal sets $X_n\to Y_n$ is a normal monomorphism.

If $X$ is a simplicial dendroidal set, we denote by
\begin{equation}\label{eq:defexponentialsegoper}
\begin{split}
\begin{aligned}
\dset& \to \op{\sset}\\
A& \longmapsto X^A
\end{aligned}
\end{split}
\end{equation}
the colimit preserving functor such that, for any tree $T$, one has the identification
\begin{equation}
X^{\Omega[T]}=X_T\, .
\end{equation}

The Boardman-Vogt tensor product of dendroidal sets (\ref{def:dendroidaltensorproduct})
induces a closed symmetric monoidal
structure on the category of simplicial dendroidal sets, for which the tensor product of
two objects $X$ and $Y$ is defined by the formula
\begin{equation}\label{eq:deftensorsimpldendsets}
(X\otimes Y)_{n}=X_n\otimes Y_n \ , \quad \text{for any $n\geq 0$.}
\end{equation}
\end{paragr}

\begin{paragr}\label{pargar:defadjsdsetpreoper}
The category of \emph{preoperads} is the full subcategory of $\sdset$
whose objects are the simplicial dendroidal sets $X$ such that the simplicial set $X_\eta$
is a discrete simplicial set (which we will view as a set). In this case, we will call $X_\eta$
the set of \emph{objects of $X$}. We write $\preoper$ for the category of preoperads.
The inclusion functor of the category of preoperads into the category of simplicial dendroidal sets
is denoted by
\begin{equation}\label{eq:inclusionpreopersdset}
\gamma^*:\preoper\to\sdset\, .
\end{equation}
It admits a left adjoint
\begin{equation}\label{eq:leftadjinclusionpreopersdset}
\gamma_!:\sdset\to\preoper
\end{equation}
as well as a right adjoint
\begin{equation}\label{eq:rightadjinclusionpreopersdset}
\gamma_*:\sdset\to\preoper
\end{equation}
(see \cite[7.2]{dend6}). We will need the explicit description
of the left adjoint $\gamma_!$, so we recall it here.
Let $X$ be a simplicial dendroidal set. Given a tree $T$, the simplicial
set $\gamma_!(X)_T$ is simply $X_T$ in the case where there does not exist
any map from $T$ to $\eta$; in the case where $T$ is linear, i.e.
admits such a map $\varepsilon:T\to\eta$
(which is necessarily unique, since $\eta$ is a terminal object in $\Omega/\eta=\Delta$),
there is the following pushout square
\begin{equation}\label{eq:explicitleftadjgamma}
\begin{split}\xymatrix{
X_\eta\ar[r]^{{\varepsilon^*}}\ar[d]&X_T\ar[d]\\
\pi_0(X_{\eta})\ar[r]&\gamma_!(X)_T
}\end{split}\end{equation}
in the category of simplicial sets.

A morphism between preoperads is a \emph{normal monomorphism} if it is
a normal monomorphism in $\sdset$.

The tensor product of simplicial dendroidal sets \eqref{eq:deftensorsimpldendsets}
preserves preoperads, so the category $\preoper$ is a closed symmetric monoidal
category. Moreover, the inclusion functor of $\dset$ into $\sdset$ factors through
the category of preoperads, and defines a colimit preserving symmetric monoidal functor
\begin{equation}\label{eq:embeddsetintopreoper}
\dset\to\preoper \ , \quad X\longmapsto X
\end{equation}
\end{paragr}

\begin{paragr}\label{def:segalcore}
Let $T$ be a tree. Let us denote by $I(T)$ the set of subobjects of $\Omega[T]$
of shape $\Omega[S]$, where $S$ is a subtree with at most one vertex
(i.e. $S$ must be either a corolla, or a copy of $\eta$). The set $I(T)$ is naturally
endowed with the structure of partially ordered set induced by the relation of inclusion
of subobjects. Another way to describe the set $I(T)$ is
\begin{equation}
I(T)=\{\text{edges of $T$}\}\amalg \{\text{vertices of $T$}\}\, .
\end{equation}
As subtrees, an edge of $T$ correspond to a map $\eta\to T$, while a vertex $v$
correspond to a map $C_{n}\to T$ which send the unique vertex of $C_{n}$ to $v$.
The partial order on $I(T)$ can be described by the property that $e<v$ if and only if
$v$ is a vertex of $T$ and $e$ is either an input edge of $v$, either the output edge of $v$.

The \emph{Segal core} of $T$ is the colimit
\begin{equation}\label{eq:defsegalcore}
\spine[T]=\varinjlim_{\Omega[S]\in I(T)}\Omega[S]
\end{equation}
in the category of dendroidal sets.
The natural map
\begin{equation}\label{eq:inclusionsegalcore}
\spine[T]\to\Omega[T]
\end{equation}
is a (normal) monomorphism: in the case where $T$ has at most vertex, we just get
$\spine[T]=\Omega[T]$, while, if $T$ has at least one inner edge,
$\spine[T]$ is the union in $\Omega[T]$ of all the
corollas $\Omega[C_{n_v}]\subset\Omega[T]$, where $v$ runs over the set of vertices of $T$,
while $C_{n_v}\to T$ denotes the corolla at $v$.

Note that, although the Segal core is not a functor, for any morphism
of trees $S\to T$ which is a composition of outer faces, there is a
canonical commutative square
\begin{equation}\label{eq:functspines}
\begin{split}\xymatrix{
\spine[S]\ar[r]\ar[d]&\Omega[S]\ar[d]\\
\spine[T]\ar[r]&\Omega[T]
}\end{split}
\end{equation}
which has the nice property of being cartesian.
\end{paragr}

\begin{prop}[{\cite[Proposition 2.4]{dend6}}]\label{prop:segcoreinneranod}
For any tree $T$, the map \eqref{eq:inclusionsegalcore} is an inner anodyne extension.
\end{prop}

\begin{definition}
A \emph{Segal operad} is a preoperad $X$ such that, for any tree $T$ with
at least one inner edge, the morphism of simplicial sets
$$X^{\Omega[T]}\to X^{\spine[T]}$$
induced by the inclusion \eqref{eq:inclusionsegalcore} is a
simplicial weak equivalence.

A Segal operad $X$ is \emph{Reedy fibrant} if, for any tree $T$, the morphism of
simplicial sets
$$X^{\Omega[T]}\to X^{\partial\Omega[T]}$$
is a Kan fibration.
\end{definition}

\begin{prop}\label{prop:charreedysegaloperads}
Let $X$ be a preoperad such that, for any tree $T$, the map
$X^{\Omega[T]}\to X^{\partial\Omega[T]}$ is a Kan fibration.
The following conditions are equivalent:
\begin{itemize}
\item[(i)] the preoperad $X$ is a Segal operad;
\item[(ii)] the preoperad $X$ is a Reedy fibrant Segal operad;
\item[(iii)]  for any tree $T$ with at least one inner edge, the map
$X^{\Omega[T]}\to X^{\spine[T]}$
is a trivial fibration;
\item[(iv)] for any tree $T$ with a given inner edge $e$, the map
$X^{\Omega[T]}\to X^{\Lambda^e[T]}$
is a trivial fibration of simplicial sets.
\end{itemize}
\end{prop}

\begin{proof}
The equivalence between conditions (i) and (ii) holds by definition, and
is stated here only for the record.
The boundary inclusions $\partial\Omega[T]\to\Omega[T]$
generate the whole class of normal monomorphisms of dendroidal sets.
Therefore, under our assumption on $X$, for any normal monomorphism
$A\to B$ in the category of dendroidal sets, the induced morphism of simplicial
sets $X^B\to X^A$ is a Kan fibration. This immediately implies that
conditions (ii) and (iii) are equivalent. The equivalence between conditions (iii)
and (iv) follows from \cite[Corollary 5.6]{dend6}.
\end{proof}

We now can summarize \cite[Theorems 8.13 \& 8.17, Remark 8.18]{dend6} as follows.

\begin{thm}\label{thm:cmcreedysegaloperads}
The category of preoperads is endowed with a left proper cofibrantly generated
symmetric monoidal model category structure for which the cofibrations are the normal
monomorphisms and the fibrant objects are the Reedy fibrant Segal operads.
Furthermore, the inclusion functor $\dset\to\preoper$ preserves and detects
weak equivalences and is a left Quillen equivalence.
\end{thm}

\begin{rem}\label{rem:genreedycofpreoper}
An explicit set of generators for the class of normal
monomorphisms of preoperads consists of the inclusion $\varnothing\to\eta$
together with all the maps of shape
$$\gamma_!(\partial\Delta[n]\times\Omega[T]\cup\Delta[n]\times\partial\Omega[T])\to
\gamma_!(\Delta[n]\times\Omega[T])$$
for any integer $n\geq 0$ and any tree $T$ \emph{with at least one vertex};
see \cite[Proposition 7.5]{dend6}.
\end{rem}

\begin{thm}\label{thm:dendroidalrezkcmc}
The category $\sdset$ of simplicial dendroidal sets
admits a model category structure, called the dendroidal Rezk model structure,
which is completely characterized by
the following properties.
\begin{itemize}
\item[(i)] The class of cofibrations is the class of (termwise)
normal monomorphisms.
\item[(ii)] If $X\to Y$ is a morphism of simplicial dendroidal sets
such that, for any integer $n\geq 0$, the morphism of dendroidal sets
$X_n\to Y_n$ is a weak equivalence of the model structure of Theorem \ref{thm:cmcdendsets},
then it is a weak equivalence.
\item[(iii)] If $X\to Y$ is a morphism of simplicial dendroidal sets
such that, for any tree $T$, the morphism of simplicial sets
$X_T\to Y_T$ is a weak equivalence of the usual Quillen model structure,
then it is a weak equivalence.
\item[(iv)] The inclusion functor $\preoper\to\sdset$ preserves and detects
weak equivalences and is a left Quillen equivalence.
\item[(v)] Any preoperad $X$ is canonically isomorphic in $\ho(\sdset)$
to the simplicial dendroidal set
$$T\longmapsto\Map(\Omega[T],X)$$
(where $\Map(-,X)$ denotes the mapping space functor
obtained by choosing a simplicial frame of a fibrant
resolution of $X$ in the sense of \cite[5.2.7]{Ho}, for instance).
\end{itemize}
\end{thm}

\begin{proof}
Properties (i), (ii) and (iii) (as well as the fact that these determine the
model structure) are summarized in \cite[Theorem 6.6]{dend6}, while
the first part of property (iv) holds by definition \cite[Definition 8.1]{dend6},
and the second part by \cite[Theorem 8.15]{dend6}.
As for property (v), we remark that,
using the Quillen equivalence $\dset\subset\preoper$ (\ref{thm:cmcreedysegaloperads}),
it is sufficient to prove it in the case of a dendroidal set; but then, this is a translation of
\cite[Propositions 3.3, 6.13 and 8.8, and Remark 8.18]{dend6}.
\end{proof}

The following two sufficient conditions for a morphism of preoperads to be a weak
equivalence will be used several times.

\begin{cor}\label{cor:termwiseweglobalwepreoper}
Let $X\to Y$ be a morphism of preoperads, and assume that one of the two
conditions below is satisfied.
\begin{itemize}
\item[(a)] For any integer $n\geq 0$, the morphism of dendroidal sets $X_n\to Y_n$
is a weak equivalence of dendroidal sets.
\item[(b)] For any tree $T$, the morphism of simplicial sets $X_T\to Y_T$
is a weak equivalence of the usual Quillen model category structure.
\end{itemize}
Then the morphism $X\to Y$ is a weak equivalence of preoperads.
\end{cor}

\begin{proof}
This follows from assertions (ii), (iii) and (iv) of the preceding theorem.
\end{proof}

\begin{prop}\label{prop:wesegoperclosedfiltcolim}
The class of weak equivalences is closed under filtered colimits
in the category of preoperads.
\end{prop}

\begin{proof}
There are many ways to prove this. For instance, one can
use the fact that a morphism of preoperads
is a weak equivalence if and only if it is a weak
equivalence of simplicial dendroidal sets for the
dendroidal Rezk model category structure;
see Theorem \ref{thm:dendroidalrezkcmc}~(iv).
But the latter is a left Bousfield localization of the generalized
Reedy model structure on $\sdset$ whose class of weak equivalences
is the class of (dendroidally) termwise simplicial weak equivalences;
see \cite[Definition 6.2]{dend6}.
As the class of simplicial weak equivalences is closed under filtered colimits,
this implies that the class of weak equivalences of the dendroidal Rezk
model structure is closed under filtered colimits as well: for any
filtered diagram $D$ of simplicial dendroidal sets, the map
from the homotopy colimit of $D$ to the colimit of $D$
is a termwise simplicial weak equivalence.
\end{proof}

\section{Weak equivalences between Segal operads}\label{sect:6}

\begin{paragr}\label{def:segalnerve}
The adjunction between the category of dendroidal sets and the
category of operads \eqref{eq:nerfdend} extends naturally
to an adjunction between the category of simplicial dendroidal sets
and the category of simplicial objects in the category of operads
(which must be distinguished from the category of simplicial operads):
\begin{equation}\label{eq:dendnervesimplicialobjects}
\tau_d:\sdset\rightleftarrows\oper^{\op{\Delta}}:\nerf_d\, .
\end{equation}
The category of simplicial operads is the full subcategory of
$\oper^{\op{\Delta}}$ which consists of presheaves $\mathcal P:\op{\Delta}\to\oper$
such that the simplicial set $\ob \mathcal P$ is discrete,
which means precisely that $\nerf_d(\mathcal P)$ is a preoperad.
Moreover, if $X$ is a preoperad, the simplicial object $\tau_d(X)$
is clearly a simplicial operad. Therefore, the
adjunction \eqref{eq:dendnervesimplicialobjects} restricts to an adjunction
\begin{equation}\label{eq:dendnervepreoperads}
\tau_d:\preoper\rightleftarrows\soper:\nerf_d
\end{equation}
in which the right adjoint $\nerf_d$ still has the property of being fully faithful.
However, the adjunction \eqref{eq:dendnervepreoperads} is not a Quillen pair
because the functor $\tau_d$ does not send normal preoperads to cofibrant simplicial operads
(in fact, if $\mathcal P$ is a $\Sigma$-cofibrant, then $X=\nerf_d(\mathcal P)$ is a normal preoperad,
but $\tau_d(X)=\mathcal P$ might not be cofibrant).
Our purpose, in this section, is to correct this defect of the model
category structure of Theorem \ref{thm:cmcreedysegaloperads} by shrinking the class of
cofibrations as much as we can. But before doing this, we have to study the notion of Segal
operad a little more closely.
\end{paragr}

\begin{rem}\label{rem:segalcondsimploper}
The functor $\nerf_d:\soper\to \preoper$ is fully faithful and
its essential image consists of the preoperads $X$
satisfying the dendroidal (strict) Segal conditions; this means that,
for any tree $T$ (with at least one inner edge), the map
$$X^{\Omega[T]}\to X^{\spine[T]}$$
is an isomorphism (this is an easy translation of \cite[Corollary 2.7]{dend6}).
In particular, for any simplicial operad $\mathcal P$, the preoperad
$\nerf_d(\mathcal P)$ is in fact a Segal operad.
\end{rem}

\begin{paragr}\label{def:operationsinsegaloper}
Let $X$ be a preoperad. Given an $(n+1)$-tuple of objects $(x_1,\ldots,x_n,x)$
of $X$ (i.e. an element of $X_\eta^{n+1}$), $n\geq 0$, we define the
simplicial set $X(x_1,\ldots,x_n;x)$ by the following pullback square
in the category of simplicial sets.
\begin{equation}
\begin{split}\xymatrix{
X(x_1,\ldots,x_n;x)\ar[r]\ar[d]&X_{C_n}\ar[d]\\
\Delta[0]\ar[r]_(.45){(x_1,\ldots,x_n,x)}&X^{n+1}_\eta
}\end{split}
\end{equation}
We remark that, as $X_\eta$ is discrete, the map
\begin{equation}\label{eq:decomposesegoperations}
\coprod_{(x_1,\ldots,x_n,x)\in X^{n+1}_\eta}X(x_1,\ldots,x_n;x)\to X_{C_n}
\end{equation}
is an isomorphism of simplicial sets.

We define a functor
\begin{equation}\label{eq:defsegaltruncation}
\pi_0:\preoper\to\dset
\end{equation}
by the formula: $\pi_0(X)_T=\pi_0(X_T)$ for any preoperad $X$ and any tree $T$.
\end{paragr}

\begin{lemma}\label{lemma:segalconditionlevelzero}
For any preoperad $X$ and any tree $T$, the natural map
$$\pi_0(X^{\spine[T]})\to\Hom_\dset(\spine[T],\pi_0(X))$$
is bijective.
\end{lemma}

\begin{proof}
This is an easy consequence of the isomorphism \eqref{eq:decomposesegoperations}
and of the fact that the functor $\pi_0:\sset\to\set$ commutes with small sums
as well as with finite products.
\end{proof}

\begin{prop}\label{prop:segalconditionlevelzero}
If $X$ is a Segal operad, then $\pi_0(X)$ is isomorphic the dendroidal nerve of an operad.
\end{prop}

\begin{proof}
We know that a dendroidal set $Y$ is isomorphic to the dendroidal nerve
of an operad if and only if it satisfies the strict dendroidal Segal condition; see \cite[Corollary 2.7]{dend6}.
If $X$ is a Segal operad, then, for any tree $T$ with at least one inner edge,
the map
$$\pi_0(X_T)=\pi_0(X^{\Omega[T]})\to\pi_0(X^{\spine[T]})$$
is bijective. Lemma \ref{lemma:segalconditionlevelzero} thus implies that
the dendroidal set $\pi_0(X)$ satisfies the strict dendroidal Segal condition.
\end{proof}

\begin{cor}\label{cor:truncationsegoper}
For any Segal operad $X$, one has canonical isomorphisms of dendroidal sets:
$\pi_0(X)\simeq\nerf_d\, \pi_0\, \tau_d(X)\simeq \nerf_d\, \tau_d\, \pi_0(X)$.
\end{cor}

\begin{proof}
By comparing the corresponding universal properties, we see that
we obviously have a natural isomorphism of operads $\tau_d(\pi_0(X))\simeq\pi_0(\tau_d(X))$,
which implies this corollary because of Proposition \ref{prop:segalconditionlevelzero}
(and using the fully faithfulness of the dendroidal nerve functor).
\end{proof}

\begin{definition}\label{def:fullyfaithfulesssurjsegoper}
Let $f:X\to Y$ be a morphism between preoperads.

The map $f$ is said to be \emph{fully faithful} if, for any $(n+1)$-tuple $(x_1,\ldots,x_n,x)$ of
objects of $X$, $n\geq 0$, the induced morphism of simplicial sets
$$X(x_1,\ldots,x_n;x)\to Y(f(x_1),\ldots,f(x_n);f(x))$$
is a weak equivalence.

The map $f$ is said to be be \emph{essentially surjective} if the morphism of operads $\pi_0(\tau_d(f))$
is essentially surjective.
\end{definition}

\begin{rem}\label{rem:truncationsegoper}
Using (the proof of) Proposition \ref{prop:segalconditionlevelzero},
it is easy to explicitly describe the operad $\pi_0(\tau_d(X))$ associated
to a given Segal operad $X$: its objects are those of $X$ (i.e. the elements of the set $X_\eta$), while
the sets of operations are given by the formula
$$\pi_0(\tau_d(X))(x_1,\ldots,x_n;x)\simeq\pi_0(X(x_1,\ldots,x_n;x))\, .$$
In particular, if a morphism between Segal operads $X\to Y$ is
fully faithful, then the morphism of operads $\pi_0(\tau_d(X))\to\pi_0(\tau_d(Y))$
is fully faithful as well. Therefore, if a morphism between Segal operads $X\to Y$ is
fully faithful and essentially surjective, then the morphism of operads $\pi_0(\tau_d(X))\to\pi_0(\tau_d(Y))$
is a weak equivalence of the model category of Theorem \ref{thm:folkcmcoperads}.
\end{rem}

\begin{prop}\label{prop:segalcoreswecorollas}
Let $f:X\to Y$ be a morphism of preoperads. The following conditions are equivalent.
\begin{itemize}
\item[(i)] The morphism $X_\eta\to Y_\eta$ is bijective, and for any $n\geq 0$,
the morphism of simplicial sets $X_{C_n}\to Y_{C_n}$ is a weak equivalence.
\item[(ii)] The morphism $f$ is fully faithful, and the morphism $X_\eta\to Y_\eta$ is bijective.
\item[(iii)] For any tree $T$, the morphism of simplicial sets $X^{\spine[T]}\to Y^{\spine[T]}$
is a weak equivalence.
\end{itemize}
\end{prop}

\begin{proof}
This follows right away from the isomorphism \eqref{eq:decomposesegoperations}:
everything comes down to the fact that the class
of weak equivalences of simplicial sets is closed under small sums, under
finite products, and under retracts.
\end{proof}

\begin{cor}\label{cor:chartermwisewesegaloper}
A morphism $X\to Y$ between Segal operads is fully faithful and
induces a bijection on objects $X_\eta\simeq Y_\eta$ if and only if, for any
tree $T$, the morphism of simplicial sets $X_T\to Y_T$ is a weak equivalence.
\end{cor}

\begin{cor}\label{cor:segaluptowe}
Let $X\to Y$ be  a morphism of preoperads. Assume that, for any tree $T$, the morphism
of simplicial sets $X_T\to Y_T$ is a weak equivalence. Then $X$ is a Segal operad
if and only if $Y$ is a Segal operad.
\end{cor}

\begin{lemma}\label{lemma:reedyfibreplacesegop}
For any Segal operad $X$, there exists a functorial morphism
of preoperads $X\to R(X)$ such that:
\begin{itemize}
\item[(a)] the preoperad $R(X)$ is a Reedy fibrant Segal operad;
\item[(b)] for any tree $T$, the morphism of simplicial sets $X_T\to R(X)_T$
is a weak equivalence.
\end{itemize}
\end{lemma}

\begin{proof}
We factor the morphism from $X$ to the terminal
preoperad by applying the small object argument
to the set $I$ of maps
$$\gamma_!(\Lambda^k[n]\times\Omega[T]\cup\Delta[n]\times\partial\Omega[T])\to
\gamma_!(\Delta[n]\times\Omega[T])$$
for $n\geq 1$, $0\leq k \leq n$, and any tree $T$.
This gives a morphism of shape $X\to R(X)$ such that the map
from $R(X)$ to the terminal preoperad has the right lifting property with respect to $I$.
Note that all the elements $A\to B$ of $I$ are normal monomorphisms
(this follows from \cite[Lemma 7.4]{dend6}), and have the property that the morphisms
$A_S\to B_S$ are weak equivalences for any tree $S$ (this latter property
comes from the explicit description of the functor $\gamma_!$ \eqref{eq:explicitleftadjgamma},
and from the cube lemma \cite[Lemma 5.2.6]{Ho}).
In particular, for any tree $T$, the morphism $R(X)^{\Omega[T]}\to R(X)^{\partial\Omega[T]}$
is a Kan fibration. The fact that $R(X)$
is a Reedy fibrant Segal operad now comes from Proposition \ref{prop:charreedysegaloperads}
and from Corollary \ref{cor:segaluptowe}.
\end{proof}

\begin{prop}\label{charwesegoperads}
A morphism between Segal operads is a weak equivalence of the model category
of Theorem \ref{thm:cmcreedysegaloperads} if and only if it is fully faithful
and essentially surjective.
\end{prop}

\begin{proof}
Let $f:X\to Y$ be a morphism of preoperads.
Using Corollaries \ref{cor:termwiseweglobalwepreoper} and
\ref{cor:chartermwisewesegaloper}, as well as Remark \ref{rem:truncationsegoper}
and Lemma \ref{lemma:reedyfibreplacesegop},
we can embed $f$ into a commutative square
$$\xymatrix{
X\ar[r]\ar[d]_f &R(X)\ar[d]^{R(f)}\\
Y\ar[r]&R(Y)
}$$
whose horizontal maps are weak equivalences, are fully faithful and bijective
on objects, and induce equivalences of operads
$$\tau_d\pi_0(X)\simeq \tau_d\pi_0(R(X))\quad \text{and}\quad 
\tau_d\pi_0(Y)\simeq \tau_d\pi_0(R(Y))\, .$$
This implies that $f$ is a weak equivalence (is fully faithful and essentiall surjective)
if and only if $R(f)$ has the same property. In other words, it is sufficient
to prove that a morphism between Reedy fibrant Segal operads is a weak equivalence
if and only if it is fully faithful and essentially surjective.
But this is already known: see \cite[Theorem 8.11]{dend6}.
\end{proof}

\begin{thm}\label{thm:truncpreoperleftquillen}
The functor $\tau_d\pi_0=\pi_0\tau_d:\preoper\to\oper$ is
a left Quillen functor which preserves weak equivalences.
\end{thm}

\begin{proof}
It is clear that this functor sends normal monomorphisms
to morphisms of operads which are injective on objects.
Therefore it is sufficient to prove that this functor
preserves weak equivalences. Note that, by Remark \ref{rem:truncationsegoper},
if $f:X\to Y$ is a morphism of Segal operads which is fully faithful and essentially
surjective, then its image by $\tau_d\pi_0=\pi_0\tau_d$ is
fully faithful and essentially surjective. In other words, by virtue of the preceding proposition,
we already know that this functor preserves weak equivalences between Segal operads.
To finish the proof, it is sufficient to construct, for any preoperad $X$, a functorial
weak equivalence into a Reedy fibrant Segal operad $X\to L(X)$ whose image by $\tau_d\pi_0=\pi_0\tau_d$
is an isomorphism. For this, we will consider two classes of morphisms.
The first class consists of morphisms
\begin{equation}\label{alpha}\tag{$\alpha$}
\gamma_!(\Lambda^k[n] \times\Omega[T]\cup \smp{n}\times\partial\Omega[T])\To
\gamma_!(\smp{n}\times\Omega[T]) 
\end{equation}
for any tree $T$ and any integers $n>0$, $0\leq k\leq n$,
while the second is the class of maps
\begin{equation}\label{beta}\tag{$\beta$}
\gamma_!(\bord\Delta[n] \times\Omega[T]\cup \smp{n}\times\Lambda^e[T])\To
\gamma_!(\smp{n}\times\Omega[T])
\end{equation}
for any tree $T$ with inner edge $e$, and any integer $n\geq 0$.
A preoperad $X$ is a Reedy fibrant Segal operad if and only if
the morphism from $X$ to the terminal object
has the right lifting property with respect to maps of shape
\eqref{alpha} and \eqref{beta}: this is a translation of
Proposition \ref{prop:charreedysegaloperads} by a standard adjunction
argument. We define $L(X)$ by applying the small object argument
to the class of maps of shapes \eqref{alpha} and \eqref{beta}.
The explicit description of the functor $\gamma_!$ \eqref{eq:explicitleftadjgamma},
\cite[Lemma 7.4]{dend6}, and the cube lemma \cite[Lemma 5.2.6]{Ho}, show that
any map of shape \eqref{alpha} is a normal monomorphism $A\to B$
such that, for any tree $S$, the map $A_S\to B_S$ is a trivial
cofibration of simplicial sets. The same arguments show that
any map of shape \eqref{beta} is a normal monomorphism $A\to B$
such that, for any integer $n\geq 0$, the map $A_n\to B_n$ is a trivial
cofibration of dendroidal sets (even an inner anodyne extension).
In particular, by Corollary \ref{cor:termwiseweglobalwepreoper},
any map of shape \eqref{alpha} or \eqref{beta} is a trivial cofibration of preoperads.
This imply that the map $X\to L(X)$ is indeed a fibrant resolution of $X$.
On the other hand, if a morphism of preoperad $A\to B$ has the property that,
for any tree $S$, the map $A_S\to B_S$ is a weak equivalence of simplicial sets, then
the morphism of operads $\tau_d\pi_0(A)\to\tau_d \pi_0(B)$ is an isomorphism,
while, if it has the property that, for any integer $n\geq 0$, the map $A_n\to B_n$ is a
weak equivalence of dendroidal sets, then the map $\pi_0\tau_d(A)\to\pi_0\tau_d(B)$
is an isomorphism (see Proposition \ref{prop:adjquillendsetoper}).
In particular, the functor $\tau_d\pi_0=\pi_0\tau_d$ sends any map
\eqref{alpha} or \eqref{beta} to an isomorphism. As this functor preserves colimits,
this implies that it sends any transfinite composition of pushouts of maps
of shape \eqref{alpha} or \eqref{beta}, such as $X\to L(X)$, to an isomorphism.
\end{proof}

\begin{rem}\label{rem:twotaud}
We have the following (essentially) commutative diagram of categories.
\begin{equation}\label{eq:twotauds1}
\begin{split}\xymatrix{
\dset\ar[rr]^{\text{inclusion}}\ar[dr]_{\tau_d}&&\preoper\ar[dl]^{ \quad \tau_d\pi_0=\pi_0\tau_d}\\
&\oper&
}\end{split}
\end{equation}
The embedding $\dset\subset\preoper$ is a left Quillen equivalence which preserves
weak equivalences (\ref{thm:cmcreedysegaloperads}). Therefore,
Proposition \ref{prop:adjquillendsetoper} and Theorem \ref{thm:truncpreoperleftquillen}
give rise to the following (essentially) commutative diagram of homotopy categories.
\begin{equation}\label{eq:twotauds2}
\begin{split}\xymatrix{
\ho(\dset)\ar[rr]^{\simeq}\ar[dr]_{\tau_d}&&\ho(\preoper)\ar[dl]^{\quad \tau_d\pi_0=\pi_0\tau_d}\\
&\ho(\oper)&
}\end{split}
\end{equation}
This has the following technical but useful consequence.
\end{rem}

\begin{cor}\label{cor:liftingisosegoperads}
Let $J_d$ be the dendroidal nerve of the contractible groupoid
with set of objects $\{0,1\}$. If $X$ is a Reedy fibrant Segal operad,
then any isomorphism in the underlying category of $\tau_d\pi_0(X)$ comes from a morphism
of preoperads $J_d\to X$.
\end{cor}

\begin{proof}
This immediately follows from the fact that $J_d$ is cofibrant and from the
commutative diagram \eqref{eq:twotauds2}, because the analogous
property for $\infty$-operads is known to hold; see \cite[Proposition 6.20]{dend3}.
\end{proof}

\section{Segal operads: the tame model category structure}\label{sect:7}

\begin{paragr}\label{paragr:enrichedfreeoperads}
For a simplicial set $K$ and a tree $T$, we denote by $T[K]$ the free
simplicial operad generated by the tree $T$ and decorated by $K$.
To be more precise, we define this first in the case where $T$ has at most
one vertex. For $T=\eta$, we just get $T[K]=\eta$, while, for $T=C_n$
a corolla, $T[K]=C_n[K]$ (see \ref{paragr:defcofandtrivcof}).
For a general tree $T$, this defines a functor
$$F(K,T):I(T)\to \soper \ , \quad (S\subset T)\longmapsto F(K,T)(S)=S[K]\, ,$$
where $I(T)$ is the partially ordered set of subtrees of $T$ with at most
one vertex; see \ref{def:segalcore}. The simplicial operad $T[K]$ is simply
the following colimit:
\begin{equation}\label{eq:defdecoratedtree}
T[K]=\varinjlim_{I(T)}F(K,T)\, .
\end{equation}
The latter can be described very concretely: its objects
are those of $T$ (i.e. the edges of $T$), while the spaces of operations
$T[K](e_1,\ldots,e_n;e)$ can be described by the formula
\begin{equation}\label{eq:explicitenrichedtreeK}
T[K](e_1,\ldots,e_n;e)=\coprod_{S\in T(e_1,\ldots,e_n;e)}K^{V(S)}\, ,
\end{equation}
where $T(e_1,\ldots,e_n;e)$ denotes the set of subtrees of $T$
whose input edges are exactly $e_1,\ldots,e_n$ and whose output edge is $e$,
and where $V(S)$ denotes the set of vertices of $S$. The composition
is defined by grafting trees and by the fact that, if a tree $S$ is obtained by
grafting the trees $S_1,\ldots,S_k$, then there is a canonical bijection
$V(S)\simeq V(S_1)\amalg\dots\amalg V(S_k)$, which induces an isomorphism
$$K^{V(S_1)}\times\dots\times K^{V(S_k)}\simeq K^{V(S)}\, .$$
We obtain a functor
\begin{equation}
\begin{split}
\begin{aligned}
\sset\times\Omega&\to\soper\\
(K,T)&\longmapsto T[K]\, .
\end{aligned}
\end{split}
\end{equation}

We also define a functor
\begin{equation}\label{eq:preoperadenrichment}
\begin{split}
\begin{aligned}
\sset\times\Omega&\to\preoper\\
(K,T)&\longmapsto\Omega[K,T]
\end{aligned}
\end{split}
\end{equation}
as follows. For a simplicial set $K$ and a tree $T$, the
simplicial dendroidal set $\Omega[K,T]$ fits into the following
pushout.
\begin{equation}\label{eq:defomegaKT}
\begin{split}\xymatrix{
\coprod_{e\in E(T)}K\times \eta\ar[r]\ar[d]&K\times \Omega[T]\ar[d]\\
\coprod_{e\in E(T)}\Delta[0]\times\eta\ar[r]&\Omega[K,T]
}\end{split}\end{equation}
For a \emph{$0$-connected simplicial set $K$}, we have the formula:
\begin{equation}\label{eq:newnotationgamma}
\Omega[K,T]\simeq\gamma_!(K\times\Omega[T]) \, .
\end{equation}
Therefore, if $K$ is $0$-connected, a map from $\Omega[K,T]$
to a preoperad $X$ is simply a morphism of simplicial sets $K\to X_T$.
In the general case, the maps out of $\Omega[K,T]$ are still understandable:
for a preoperad $X$, the datum of a map $\Omega[K,T]\to X$
is equivalent to a pair $(v,f)$, where $v$
is a function from the set of edges $E(T)$ to $X_\eta$,
while $f$ is a morphism of simplicial sets from $K$ to $X_T$,
such that the following square commutes.
$$\xymatrix{
K\ar[r]^f\ar[d]&X_T\ar[d]\\
\Delta[0]\ar[r]^v&X_\eta^{E(T)}
}$$

Finally, these constructions may be compared as follows.
\end{paragr}

\begin{prop}\label{prop:computenerveenrichedtrees}
For any simplicial set $K$ and any tree $T$,
there is a canonical isomorphism of operads
$$\tau_d(\Omega[K,T]) \simeq T[K]\, .$$
If moreover $T$ has at most one vertex, then
the unit map
$$\Omega[K,T]\to\nerf_d(T[K])$$
is an isomorphism.
\end{prop}

\begin{proof}
Let us prove first that, for any tree $T$ with at most one vertex, we have a canonical isomorphism
$\Omega[K,T]\simeq \nerf_d(T[K])$.
If $T=\eta$, there is nothing to check, while, if $T=C_n$ is a corolla ($n\geq 0$),
we proceed as follows. Given a tree $S$, we see (by considering first the case where $K$
is a set with a single element, and then by using the pushout square \eqref{eq:defomegaKT}) that
$$\Omega[K,C_n]_S=\begin{cases}
K\quad \text{if $S$ is obtained by grafting linear trees with $C_n$ and $n\neq 1$,}\\
\ob C_n\amalg(K\times \{\text{surjections $S\to C_n$}\}) \quad \text{if $S$ is a linear tree,}\\
\varnothing \quad \text{otherwise.}
\end{cases}$$
It is easy to see that the set $\nerf_d(C_n[K])_S$ admits the same description, which
means that we have natural isomorphisms $\Omega[K,C_n]\simeq \nerf_d(C_n[K])$.
The inclusions $\spine[T]\to\Omega[T]$ being sent by $\tau_d$ to isomorphims,
we see that we have a natural isomorphism
$$\varinjlim_{S\in I(T)} \tau_d(\Omega[K,S])\simeq\tau_d(\Omega[K,T])\, ,$$
which ends the proof, by construction of $T[K]$ \eqref{eq:defdecoratedtree}.
\end{proof}

\begin{paragr}\label{paragr:enrichedsegalcores}
For a simplicial set $K$ and a tree $T$, we define the \emph{Segal core of $\Omega[K,T]$}
as the following colimit.
\begin{equation}\label{eq:defenrichedsegalcore}
\spine[K,T]=\varinjlim_{S\in I(T)}\Omega[K,S]
\end{equation}
We immediately see that we have the following canonical pushout
of simplicial dendroidal sets.
\begin{equation}\label{eq:pushoutspineKT}
\begin{split}\xymatrix{
\coprod_{e\in E(T)}K\times \eta\ar[r]\ar[d]&K\times \spine[T]\ar[d]\\
\coprod_{e\in E(T)}\Delta[0]\times\eta\ar[r]&\spine[K,T]
}\end{split}\end{equation}
With this description, it follows that the canonical map
\begin{equation}\label{eq:embedenrichedsegalcores}
\spine[K,T]\to\Omega[K,T]
\end{equation}
is a monomorphism between normal preoperads.
For any monomorphism of simplicial sets $K\to L$, one obtains
a commutative square
\begin{equation}\label{eq:fctembedenrichedsegalcores}
\begin{split}
\xymatrix{
\spine[K,T]\ar[r]\ar[d]&\Omega[K,T]\ar[d]\\
\spine[L,T]\ar[r]&\Omega[L,T]
}\end{split}
\end{equation}
which is easily seen to be cartesian. This induces a map
\begin{equation}\label{eq:fctembedenrichedsegalcoresbis}
\spine[L,T]\cup\Omega[K,T]\to\Omega[L,T]
\end{equation}
which is also a monomorphism between normal preoperads.
\end{paragr}

\begin{prop}\label{prop:enrichedsegalcorestrivialcof}
The maps \eqref{eq:embedenrichedsegalcores}
and \eqref{eq:fctembedenrichedsegalcoresbis} are
trivial cofibrations.
\end{prop}

\begin{proof}
It is clearly sufficient to treat the case of the map \eqref{eq:embedenrichedsegalcores}.
By applying the cube lemma \cite[Lemma 5.2.6]{Ho} to the obvious
morphism of squares from \eqref{eq:defomegaKT} to \eqref{eq:pushoutspineKT},
this proposition becomes a direct consequence of
Proposition \ref{prop:segcoreinneranod} and Corollary \ref{cor:termwiseweglobalwepreoper}.
\end{proof}

\begin{rem}\label{rem:taudenrichedsegalcores}
The functor $\tau_d$ sends the map \eqref{eq:embedenrichedsegalcores} to
an isomorphism of simplicial operads $\tau_d\spine[K,T]\simeq\tau_d\Omega[K,T]=T[K]$:
this follows from Proposition \ref{prop:computenerveenrichedtrees}
and from formula \eqref{eq:defdecoratedtree}. Therefore, the map
\eqref{eq:fctembedenrichedsegalcoresbis} is also sent by $\tau_d$ to an isomorphism.
\end{rem}

\begin{prop}\label{prop:niceenrichedsegalcores}
For any tree $T$ any integer $n\geq 0$, there is
a canonical isomorphism
$$\gamma_!(\Delta[n]\times\spine[T]\cup\partial\Delta[n]\times\Omega[T])\simeq
\spine[\Delta[n],T]\cup\Omega[\partial\Delta[n],T]\, .$$
\end{prop}

\begin{proof}
This is clear whenever $n\geq 2$ by Formula \ref{eq:newnotationgamma}.
If $n=0$, then $\Omega[\partial\Delta[n]]$ is contained in $\spine[\Delta[n],T]$,
and we are reduced to prove that we have an isomorphism
of shape $\gamma_!((\Delta[n]\times\spine[T])\simeq\spine[\Delta[n],T)$,
which is true again by \ref{eq:newnotationgamma}.
It remains to understand the case where $n=1$.
We have a canonical commutative square
$$\xymatrix{
\gamma_!(\Delta[n]\times\spine[T]\cup\partial\Delta[n]\times\Omega[T])\ar[r]\ar[d]&
\spine[\Delta[n],T]\cup\Omega[\partial\Delta[n],T]\ar[d]\\
\gamma_!(\Delta[n]\times\Omega[T])\ar[r]&\Omega[\Delta[n],T]
}$$
The left vertical map is
a monomorphism (this follows from \cite[Lemma 7.4]{dend6}, for instance). 
On the other hand, the lower horizontal map is known to be
an isomorphism (\ref{eq:newnotationgamma} again).
Therefore, the upper horizontal map is a monomorphism, and it remain to
check that it is an epimorphism for $n=1$. In other words, we have to check that
we have a surjection of shape $\Omega[T]\amalg\Omega[T]\to\Omega[\partial\Delta[1],T]$,
which is obvious.
\end{proof}

\begin{paragr}\label{paragr:deftamecof}
We now define four new classes of morphisms of preoperads.
The class $\tc$ contains:
\begin{itemize}
\item[(TC1)] the map $\varnothing\to\eta$;
\item[(TC2)] the inclusions $\Omega[\partial\Delta[n],T]\to\Omega[\Delta[n],T]$
for any integer $n\geq 0$ and any tree $T$ with a unique vertex;
\item[(TC3)] the  maps of shape $\spine[\Delta[n],T]\cup\Omega[\partial\Delta[n],T]\to\Omega[\Delta[n],T]$
for any integer $n\geq 0$ and any tree $T$.
\end{itemize}
The class $\otc$ is the saturation of $\tc$. The class $\ta$ contains:
\begin{itemize}
\item[(TA1)] the maps of shape $\eta\to H$ which
are weak equivalences, and such that $H$ is countable, with set of objects $H_\eta=\{0,1\}$,
while the map $\{0\}\amalg\{1\}\to H$ is in $\otc$
(a preoperad $X$ is countable if each of the sets $X_{n,T}$ is countable);
\item[(TA2)] the inclusions $\Omega[\Lambda^k[n],T]\to\Omega[\Delta[n],T]$
for $n\geq 1$, $0\leq k\leq n$,  and any tree $T$ with a unique vertex;
\item[(TA3)] the  maps of shape $\spine[\Delta[n],T]\cup\Omega[\partial\Delta[n],T]\to\Omega[\Delta[n],T]$
for any integer $n\geq 0$ and any tree $T$.
\end{itemize}
The class $\ota$ is the saturation of the class $\ta$.

Note that condition (TA3) is identical to condition (TC3).
\end{paragr}

\begin{definition}
A morphism of preoperads is a \emph{tame cofibration} if it belongs to the class $\otc$.

A morphism of preoperads is an \emph{isofibration} if it has the right lifting property with respect to
$\ota$.

A preoperad $X$ is a \emph{fibrant Segal operad} if the morphism from $X$ to the terminal object
is an isofibration.
\end{definition}

\begin{prop}\label{prop:analysetameRLP}
Let $f:X\to Y$ be a morphism of preoperads.
\begin{itemize}
\item[(a)] The morphism $f$ has the right lifting property with respect to the
inclusion $\varnothing\to\eta$ if and only if it is surjective on objects.
\item[(b)] Given a morphism of simplicial sets $u:K\to L$, the map $f$
has the right lifting property with respect to the induced maps
$\Omega[K,T]\to\Omega[L,T]$, for any tree $T$ with a unique vertex,
if and only if, for any $(n+1)$-tuple $(x_1,\ldots,x_n,x)$ of objects of $X$, $n\geq 0$,
the morphism
$$X(x_1,\ldots,x_n;x)\to Y(f(x_1),\ldots,f(x_n);f(x))$$
has the right lifting property with respect to the map $u$.
\item[(c)] The morphism $f$ has the right lifting property with
respect to the inclusions $\spine[\Delta[n],T]\cup\Omega[\partial\Delta[n],T]\to\Omega[\Delta[n],T]$,
for any integer $n\geq 0$ and any tree $T$, if and only if, the morphism
$X^{\Omega[T]}\to Y^{\Omega[T]}\times_{Y^{\spine[T]}}X^{\spine[T]}$
is a trivial fibration for any tree $T$ with at least one inner edge.
\end{itemize}
\end{prop}

\begin{proof}
Assertion (a) is obvious, while (b) follows from the construction of $\Omega[K,T]$
(see the end of \ref{paragr:enrichedfreeoperads}).
Assertion (c) follows from Formula \ref{eq:newnotationgamma}
and Proposition \ref{prop:niceenrichedsegalcores} (note that $\spine[T]=\Omega[T]$
whenever $T$ has at most one vertex).
\end{proof}

\begin{cor}\label{cor:tametrivifibsegoper}
A morphism of preoperads $f:X\to Y$ has the right lifting property with respect to
tame cofibrations if and only if the following three conditions are verified:
\begin{itemize}
\item[(i)] it is surjective on objects;
\item[(ii)] for any $(n+1)$-tuple $(x_1,\ldots,x_n,x)$ of objects of $X$, $n\geq 0$,
the morphism $X(x_1,\ldots,x_n;x)\to Y(f(x_1),\ldots,f(x_n);f(x))$ is a trivial fibration;
\item[(iii)] the map $X^{\Omega[T]}\to Y^{\Omega[T]}\times_{Y^{\spine[T]}}X^{\spine[T]}$
is a trivial fibration for any tree $T$ with at least one inner edge.
\end{itemize}
In particular, by Proposition \ref{charwesegoperads},
if $X$ and $Y$ are Segal operads, such a morphism is a weak equivalence.
\end{cor}

\begin{proof}
By virtue of the preceding proposition,
conditions (i), (ii) and (iii) correspond to the right lifting property
with respect to maps of shape (TC1), (TC2) and (TC3), respectively.
\end{proof}

\begin{cor}\label{cor:charfibsegalop}
A preoperad $X$ is a fibrant Segal operad if and only if it satisfies the following two conditions:
\begin{itemize}
\item[(i)]  for any $(n+1)$-tuple $(x_1,\ldots,x_n,x)$ of objects of $X$, $n\geq 0$,
the simplicial set $X(x_1,\ldots,x_n;x)$ is a Kan complex;
\item[(ii)] for any tree $T$ with at least one inner edge, the morphism
$X^{\Omega[T]}\to X^{\spine[T]}$ is a trivial fibration.
\end{itemize}
In particular, any fibrant Segal operad is a Segal operad.
\end{cor}

\begin{proof}
Here again, conditions (i) and (ii) express the property
of the map from $X$ to the terminal object of having
the right lifting property with respect to maps of shapes
(TA2) and (TA3), respectively. Note that, in that case,
the right lifting property with respect to maps of shape
(TA1) is automatic: indeed, $H$ must have the property
that $H_T=\varnothing$ for any non linear tree $T$,
which imply that it is isomorphic to $\eta\times H$; in other words,
all the maps of shape (TA1) must have a retraction.
\end{proof}

\begin{cor}
A simplicial operad $\mathcal P$ is fibrant if and only if its dendroidal nerve $\nerf_d(\mathcal P)$
is a fibrant Segal operad.
\end{cor}

\begin{paragr}\label{paragr:changeobjectspreoper}
Let $X$ be a preoperad, and $f:S\to X_\eta$
a map from a set $S$ to the set of objects of $X$.
We define the preoperad $f^*(X)$ as the one whose dendrices
are the simplicial sets obtained by the pullbacks
\begin{equation}\label{eq:defchangeobjectspreoper}
\begin{split}
\xymatrix{
f^*(X)_T\ar[r]\ar[d]&X_T\ar[d]\\
S^{E(T)}\ar[r]^f&X_\eta^{E(T)}
}\end{split}
\end{equation}
in which $E(T)$ denotes the set of edges of $T$, while
the right vertical arrow is the obvious evaluation map.
We have a canonical morphism of preoperads
\begin{equation}\label{eq:compmapchangeobjpreoper}
X\to f^*(X)\, .
\end{equation}
Any morphism of preoperads $u:X\to Y$
factors into a morphism $X\to u_\eta^*(Y)$
which is bijective on objects followed by the
canonical map $u^*_\eta(Y)\to Y$.
\end{paragr}

\begin{lemma}\label{lemma:surjchangeobjuniveq}
Let $X$ be a preoperad. For any surjective map
$f:S\to X_\eta$, the induced morphism $f^*(X)\to X$
is a universal weak equivalence (i.e. it is a weak
equivalence which remains so after any base change).
\end{lemma}

\begin{proof}
For any integer $n\geq 0$, it is easy to see that
the morphism of dendroidal sets $f^*(X)_n\to X_n$
has the right lifting property with respect to the boundary inclusions
$\partial\Omega[T]\to\Omega[T]$ (observe that, except in the case where $T=\eta$,
these inclusions are bijective on objects, so that the lifts are in fact unique).
In other words, this map is a termwise trivial fibration
of dendroidal sets. Corollary \ref{cor:termwiseweglobalwepreoper} ends
the proof.
\end{proof}

\begin{prop}\label{prop:tametrivialfib}
Any morphism with the right lifting property with respect to
tame cofibrations is a weak equivalence.
\end{prop}

\begin{proof}
Let $f:X\to Y$ be a morphism with the right lifting property with respect to
tame cofibrations. In particular, $f$ is surjective on objects, so that,
by virtue of Lemma \ref{lemma:surjchangeobjuniveq}, the map $f^*_\eta(Y)\to Y$
is a weak equivalence. On the other hand, we see from Corollary \ref{cor:tametrivifibsegoper}
that the map $X\to f^*_\eta(Y)$ still has the right lifting property with respect to tame cofibrations:
the only apparently non trivial property to check is condition (iii), which follows
from the fact that the maps
$$f^*_\eta(Y)^{\Omega[T]}\times_{f^*_\eta (Y)^{\spine[T]}}X^{\spine[T]}
\to Y^{\Omega[T]}\times_{Y^{\spine[T]}}X^{\spine[T]}$$
are bijective.
In other words, we may assume that $f$ is bijective on objects. We see that, for any tree $T$
with at least one vertex, the morphism of simplicial sets $X_T\to Y_T$ is
a trivial fibration (if $T$ is a corolla, this follows from condition (ii)
of Corollary \ref{cor:tametrivifibsegoper}; this property extends to any tree $T$ with at least
one inner edge, by condition (iii) of the same corollary, and from the fact that the class
of trivial fibrations of simplicial sets is closed under small sums).
Corollary \ref{cor:termwiseweglobalwepreoper} thus implies that $f$ is a weak equivalence.
\end{proof}

\begin{prop}\label{prop:necessaryisofib}
Let $Y$ be a Segal operad. If a morphism $f:X\to Y$ has the right lifting
property with respect to maps of type (TA2) and (TA3),
it must satisfy the following three conditions:
\begin{itemize}
\item[(i)] the preoperad $X$ is a Segal operad;
\item[(ii)] for any $(n+1)$-tuple $(x_1,\ldots,x_n,x)$ of objects of $X$, $n\geq 0$,
the morphism $X(x_1,\ldots,x_n;x)\to Y(f(x_1),\ldots,f(x_n);f(x))$ is a Kan fibration;
\item[(iii)] for any tree $T$, the map $X^{\Omega[T]}\to Y^{\Omega[T]}\times_{Y^{\spine[T]}}X^{\spine[T]}$
is a trivial fibration.
\end{itemize}
\end{prop}

\begin{proof}
We already know from Proposition \ref{prop:analysetameRLP}
that conditions (ii) and (iii) are equivalent to the
right lifting property with respect to maps of shape (TA2) and (TA3), respectively.
Let us prove that $X$ must be a Segal operad.
Let $T$ be a tree with at least one inner edge. We then have the following
commutative square.
$$\xymatrix{
X_T\ar[r]\ar[d]&Y_T\ar[d]\\
X^{\spine[T]}\ar[r]&Y^{\spine[T]}
}$$
The bottom map is a Kan fibration: as the class of Kan
fibrations is closed under small sums and small products, this follows
from condition (ii) and isomorphisms of shape \eqref{eq:decomposesegoperations}.
Moreover, as $Y$ is a Segal operad, the right vertical map is a weak equivalence.
As the model category of simplicial sets is right proper, this implies that the projection
$Y^{\Omega[T]}\times_{X^{\spine[T]}}Y^{\spine[T]}\to X^{\spine[T]}$ is a
weak equivalence. Therefore, condition (iii) implies that the left vertical map is a weak equivalence,
or, in other words, that $X$ is a Segal operad.
\end{proof}

\begin{prop}\label{prop:isofibsegopimpliesisofiboper}
Let $Y$ be a fibrant Segal operad. If a morphism $f:X\to Y$ is an isofibration,
then the induced morphism of operads $\tau_d\, \pi_0(X)\to \tau_d\, \pi_0(Y)$ is an isofibration.
\end{prop}

\begin{proof}
Let $y_0\simeq y_1$ be an isomorphism
in the category underlying the operad $\tau_d\, \pi_0(Y)$, and let $x_0$ be an object
of $X$ such that $f(x_0)=y_0$.
By virtue of Lemma \ref{lemma:reedyfibreplacesegop},
we can choose a Reedy fibrant resolution $u:Y\to Y'$
such that $u$ is a weak equivalence which is bijective on objects,
and $Y'$ is a Reedy fibrant Segal operad. 
Let $K$ be a 
Segal operad with set of objects $K_\eta=\{0,1\}$
such that the inclusion $\eta\simeq\{0\}\to K$ is a normal monomorphism
and a weak equivalence (for instance, $K$ might be choosen to be
the nerve of the contractible groupoid with set
of objects $\{0,1\}$). As $K$ is cofibrant and $Y'$ Reedy fibrant, the isomorphism
$y_0\simeq y_1$ may be lifted to a map $K\to Y'$ sending $e$ to $y_e$
for $e=0,1$ (Corollary \ref{cor:liftingisosegoperads}).
We can apply the small object argument to the class of maps
of shapes (TA2) and (TA3) to factor the map $K\to Y$
into a weak equivalence $K\to K'$ followed by a morphism $K'\to Y'$
which is bijective on objects and has the right lifting with respect to maps of shapes (TA2) and (TA3).
Proposition \ref{prop:necessaryisofib} implies that $K'$ is a fibrant
Segal operad and that, for any tree $T$, the map $K'_T\to Y'_T$ is a Kan fibration. 
It is then time to form the following
pullback square.
$$\xymatrix{
H'\ar[r]\ar[d]&K'\ar[d]\\
Y\ar[r]&Y'
}$$
As the right hand morphism is a termwise Kan fibration,
while the lower horizontal map is a termwise simplicial
weak equivalence (\ref{cor:chartermwisewesegaloper}), the
upper horizontal map is a termwise simplicial weak equivalence,
whence a weak equivalence (\ref{cor:termwiseweglobalwepreoper}).
Let us fix a countable tame cofibration $\eta\amalg\eta\to H$ which is a bijection
on objects, and such that there exists a weak equivalence $H\to\eta$:
one obtains this by applying the small object argument to the class of
maps of shapes (TC2) and (TC3) to factor the codiagonal $\eta\amalg\eta\to\eta$;
the fact that $H$ may be assumed to be countable comes from the fact that
any map of shape (TC2) or (TC3) has a finite codomain, while the fact
that $H\to\eta$ is a weak equivalence comes from Proposition \ref{prop:tametrivialfib}.
Remark that $H'$ is a fibrant Segal operad by Proposition \ref{prop:necessaryisofib}.
This implies that the map $H'\simeq \eta\times H'\to \eta$
is fully faithul and surjective on objects. In particular, this map
has the right lifting property with respect to tame cofibrations (\ref{cor:tametrivifibsegoper}).
Therefore, we can choose a map $H\to H'$ which is bijective on objects.
Composing with $H'\to Y$, we thus obtain a map $y:H\to Y$ that lifts the isomorphism $y_0\simeq y_1$. 
Finally, we see that we have constructed a commutative square
$$\xymatrix{
\eta\ar[r]^{x_0}\ar[d]_0&X\ar[d]^f\\
H\ar[r]^y&Y
}$$
and we would like it to admit a lift from $H$ to $X$.
But, as $H$ is countable, the left vertical arrow is a map of type (TA1), so that
such a lift $h:H\to X$ exists by assumption on $f$. The map $h$ induces an isomorphism
$x_0\simeq x_1$ in $\tau_d\, \pi_0(X)$ which lifts the isomorphism $y_0\simeq y_1$.
\end{proof}

\begin{cor}\label{cor:isofibrationsfibsegopfib}
A morphism between fibrant Segal operads is both an isofibration and a weak equivalence if and only if
it has the right lifting property with respect to tame cofibrations.
\end{cor}

\begin{proof}
Any morphism between Segal operads with the right lifting property
with respect to tame cofibrations is certainly an isofibration as well
a weak equivalence; see Corollary \ref{cor:tametrivifibsegoper}.
Let $f:X\to Y$ be an isofibration between fibrant Segal operads which is also
a weak equivalence. As $f$ is fully faithful, we immediately see that
$f$ satisfies conditions (ii) and (iii) of Corollary \ref{cor:tametrivifibsegoper}.
To prove that $f$ has the right lifting property with respect to tame
cofibrations, it is thus sufficient to check that it is surjective on objects.
But, by virtue of the preceding proposition,
the induced morphism of operads $\tau_d\, \pi_0(f)$ is then an isofibration,
and as it is clearly a weak equivalence, it must be surjective on objects.
As the functor $\tau_d\, \pi_0$ does not affect the objects, this means that
$f$ is surjective on objects as well.
\end{proof}

\begin{thm}\label{thm:tamemodelcat}
The category of preoperads is endowed with a left proper cofibrantly generated model
category structure whose cofibrations are the tame cofibrations, while the fibrant objects
are the fibrant Segal operads and the fibrations between Segal operads are the isofibrations.
Moreover, the identity functor is a left Quillen equivalence from this model category
to the model category of Theorem \ref{thm:cmcreedysegaloperads}.
\end{thm}

\begin{proof}
As any tame cofibration is a normal monomorphism, the class of
tame cofibrations which are also weak equivalences is closed
under pushout, retracts, and transfinite composition.
Therefore, the category of preoperads admits a cofibrantly generated model category
structure whose weak equivalences are the same as those of the model
category structure of Theorem \ref{thm:cmcreedysegaloperads}, and whose
cofibrations are precisely the tame cofibrations: this immediately follows
from Proposition \ref{prop:tametrivialfib} and from Jeff Smith's theorem
\cite[Theorem 1.7 and Proposition 1.18]{beke}. Moreover, it is clear that
the class $\ota$ is contained in the class of trivial cofibrations (the only
non obvious case is about maps of type (TA3), which is solved by
Proposition \ref{prop:enrichedsegalcorestrivialcof}). Therefore, any
fibrant object is a fibrant Segal operad, and any fibration is an isofibration.
We will now prove that any isofibration between fibrant Segal operads
is a fibration for this model category structure.
Consider a commutative square
$$\xymatrix{
A\ar[r]^a\ar[d]_i&X\ar[d]^f\\
B\ar[r]^b&Y
}$$
where $i$ is a tame cofibration and a weak equivalence, while $f$ is
an isofibration between fibrant Segal operads.
We want to produce a lift from $B$ to $X$. We can factor the map $b$ into a morphism
$k:B\to B'$ in $\ota$ and an isofibration $b':B'\to Y$.
It is then sufficient to prove that the commutative
square
$$\xymatrix{
A\ar[r]^a\ar[d]_{i'}&X\ar[d]^f\\
B'\ar[r]^{b'}&Y
}$$
admits a lift, where $i'=ki$. For this, it is sufficient to prove that
the map $i'$ is in the class $\ota$.
We can choose a factorization of $i'$
as a map $j:A\to C$ in $\ota$ followed by
an isofibration $q:C\to B'$. But then, the map $q$
is an isofibration between fibrant Segal operads and a weak equivalence.
Therefore, by virtue of Corollary \ref{cor:isofibrationsfibsegopfib},
the morphism $q$ has the right lifting property with respect to
any tame cofibration. This implies that the tame cofibration $i'$ is
a retract of $j$, whence is in $\ota$.
The property of left properness follows from the analogous property
for the model category structure of Theorem \ref{thm:cmcreedysegaloperads}
and from the fact that any tame cofibration is in particular a normal
monomorphism. The last assertion of the theorem is obvious.
\end{proof}

\section{Strictifications}\label{sect:8}

In this section we will prove that the homotopy theories of $\infty$-operads,
of Segal operads, and of simplicial operads are all canonically Quillen equivalent.
We will start by comparing Segal operads and simplicial operads.

\begin{prop}\label{prop:taudpreservestamecof}
The functor $\tau_d:\preoper\to\soper$ sends tame cofibrations
to cofibrations.
\end{prop}

\begin{proof}
It is sufficient to check that the functor $\tau_d$
sends maps of type (TC1), (TC2), or (TC3) (cf. \ref{paragr:deftamecof})
to cofibrations. The case of (TC1) is clear, while maps of type
(TC2) are sent to maps of type (C2) (this follows from Proposition \ref{prop:computenerveenrichedtrees}).
As for the maps of type (TC3), we know from Remark \ref{rem:taudenrichedsegalcores} that they become
isomorphisms in the category of simplicial operads.
\end{proof}

\begin{lemma}\label{lemma:pushoutnervepreoper}
Consider a pushout square of preoperads
$$\xymatrix{
A\ar[r]\ar[d]_i&U\ar[d]^f\\
B\ar[r]&V
}$$
in which $i$ is a map of type (TC1), (TC2), or (TC3) (see \ref{paragr:deftamecof}).
If the simplicial operad $\tau_d(U)$ is $\Sigma$-cofibrant, then
so is $\tau_d(V)$, and the comparison morphism
$$V\amalg_U\nerf_d\, \tau_d(U)\to \nerf_d\, \tau_d(V)$$
is a weak equivalence.
\end{lemma}

\begin{proof}
We have a cocartesian square
$$\xymatrix{
\tau_d(A)\ar[r]\ar[d]_{\tau_d(i)}&\tau_d(U)\ar[d]^{\tau_d(f)}\\
\tau_d(B)\ar[r]&\tau_d(V)
}$$
in which the morphism $\tau_d(i)$ is a cofibration (by the preceding proposition),
so that the morphism $\tau_d(f)$ is a cofibration. As $\tau_d(U)$
is $\Sigma$-cofibrant, this implies that $\tau_d(f)$ is a $\Sigma$-cofibration
and that the simplicial operad $\tau_d(V)$ is $\Sigma$-cofibrant.
Let us write $\overline V=V\amalg_U\nerf_d\, \tau_d(U)$.
It remains to prove that $\overline V\to \nerf_d\, \tau_d(V)$ is a weak equivalence.

The case where $i$ is of type (TC1) is clear: the functors
$\tau_d$ and $\nerf_d$ don't affect the objects and preserve sums, so that
we must have $\overline V=\eta\amalg\nerf_d(U)\simeq\nerf_d(\eta\amalg U)=\nerf_d(V)$.
The case where $i$ is of type (TC3) is easy to understand too: in this case, we have
a commutative square of shape
$$\xymatrix{
U\ar[r]\ar[d]_f&\nerf_d\, \tau_d(U)\ar[d]^{\nerf_d\, \tau_d(f)}\\
V\ar[r]&\nerf_d\, \tau_d(V)
}$$
in which the left vertical map is a trivial cofibration
(because it belongs to $\ota$), while the right vertical map
is an isomorphism (by Remark \ref{rem:taudenrichedsegalcores}, the map $\tau_d(f)$ is
a pushout of an isomorphism). Therefore the map $\overline V\to \nerf_d\, \tau_d(V)$
has a section which is a trivial cofibration, whence it is a weak equivalence.

It remains to consider the case where $i$ is a map of type (TC2).
More precisely, we have to understand pushout squares of shape
$$\xymatrix{
\Omega[\partial\Delta[m],C_n]\ar[r]\ar[d]_i&U\ar[d]^f\\
\Omega[\Delta[m],C_n]\ar[r]&V
}$$
for $m,n\geq 0$.
By virtue of Corollary \ref{cor:termwiseweglobalwepreoper}, it is sufficient to prove that,
for any integer $q\geq 0$, the morphism of dendroidal sets
$$\overline V_q\to \nerf_d\, \tau_d(V)_q=\nerf_d\, \tau_d(V_q)$$
is a weak equivalence of the model category structure of Theorem \ref{thm:cmcdendsets}.
Let us denote by $I_q$ the set of all surjective maps $\Delta[q]\to\Delta[m]$.
We then have pushout squares of the following type in the category of operads.
$$\xymatrix{
\coprod_{f\in I_q}\partial C_n\ar[r]\ar[d]&C_n[\partial\Delta[m]]_q\ar[d]\ar[r]&\tau_d(U_q)\ar[d]\\
\coprod_{f\in I_q}C_n\ar[r]&C_n[\Delta[m]]_q\ar[r]&\tau_d(V_q)
}$$
Let us form the following cocartesian squares in the category of dendroidal sets.
$$\xymatrix{
\coprod_{f\in I_q}\partial\Omega[C_n]\ar[r]\ar[d]&\Omega[\partial\Delta[m],C_n]_q\ar[d]\ar[r]
&U_q\ar[r]&\nerf_d\, \tau_d(U_q)\ar[d]\\
\coprod_{f\in I_q}\Omega[C_n]\ar[r]&S\ar[rr]&&W
}$$
Note that, by virtue of the last assertion of Proposition \ref{prop:computenerveenrichedtrees},
we have $\Omega[K,C_n]=\nerf_d(C_n[K])$ for any simplicial set $K$.
We know that the comparison map
$$S\to \nerf_d\, \tau_d(S)\simeq \nerf_d(C_n[\Delta[m]]_q)=\Omega[\Delta[m],C_n]_q$$
is an inner anodyne extension: this follows by applying Corollary \ref{cor:addeverelfreeells} to the left hand square.
By applying Corollary \ref{cor:addeverelfreeells} to the composed square
(which makes sense because we assumed $\tau_d(U_q)$ to be $\Sigma$-cofibrant),
we obtain that the map
$$W\to \nerf_d\, \tau_d(W)\simeq \nerf_d\, \tau_d(V)_q$$
is an inner anodyne extension as well. We next observe that we have the pushout
$$\xymatrix{
S\ar[r]\ar[d]&W\ar[d]\\
\Omega[\Delta[m],C_n]_q\ar[r]&\overline V_q
}$$
from which we deduce that the map $W\to \overline V_q$ is an inner anodyne extension.
Finally, the commutative triangle
$$\xymatrix{
W\ar[dr]\ar[d]&\\
\overline V_q\ar[r]&\nerf_d\, \tau_d(V_q)
}$$
shows that the map $\overline V_q \to \nerf_d\, \tau_d(V_q)$ is a weak equivalence.
\end{proof}

\begin{prop}\label{prop:unitpreoperwe}
Let $X$ be a preoperad. If $X$ is tamely cofibrant (i.e. if the map $\varnothing\to X$
is a tame cofibration), then the unit map $X\to \nerf_d\, \tau_d(X)$ is a weak
equivalence.
\end{prop}

\begin{proof}
By the small object argument, we may assume that
$X$ is a retract of an object $Y$, such that $Y=\bigcup Y_i$,
where
$$\varnothing=Y_0\subset Y_1\subset \dots \subset Y_i \subset\dots \subset Y\ , \quad i\in I\, ,$$
is a sequence of maps indexed by a well ordered set $I$ with initial element $0$,
such that each of the inclusions
$Y'_i=\varinjlim_{j<i}Y_j\to Y_i$, $i\in I$,
fits into a pushout of shape
$$\xymatrix{
A_i\ar[r]\ar[d]_{u_i}&Y'_i\ar[d]\\
B_i\ar[r]&Y_{i}
}$$
where the map $u_i$ is a morphism of type (TC1), (TC2), or (TC3) (\ref{paragr:deftamecof}).
Note that the functors $\tau_d$ and $\nerf_d$ both preserve small
filtered colimits. As the class of weak equivalences is closed under retracts
and under filtered colimits (\ref{prop:wesegoperclosedfiltcolim}),
it is sufficient to prove that the unit map 
$Y_i\to \nerf_d\, \tau_d(Y_i)$
is a weak equivalence for each $i\in I$.
We do this by transfinite induction.
For $i=0$, we must have $Y'_0=\varnothing$, which implies that $u_0$
is of type (TC1). Therefore, $Y_0=\eta$, and the map
$Y_0\to \nerf_d\, \tau_d(Y_0)$ is an isomorphism.
It remains to study the case where $i>0$.
Proposition \ref{prop:taudpreservestamecof} implies that,
for any tamely cofibrant preoperad $W$, the simplicial
category $\tau_d(W)$ is cofibrant, whence $\Sigma$-cofibrant.
We deduce from Lemma \ref{lemma:pushoutnervepreoper}
that the comparison map
$Y_{i}\amalg_{Y'_{i}} \nerf_d\, \tau_d(Y'_{i})\to \nerf_d\, \tau_d(Y_i)$
is a weak equivalence. Moreover, the map $Y'_{i}\to \nerf_d\, \tau_d(Y'_{i})$
is a weak equivalence, because it is a filtered colimit of the weak equivalences
$Y_j\to \nerf_d\, \tau_d(Y_j)$, $j<i$. As the model category of
preoperads is left proper, and as the map $Y'_{i}\to Y_i$
is a cofibration, the map
$Y_i\to Y_{i}\amalg_{Y'_{i}} \nerf_d\, \tau_d(Y'_{i})$
is a weak equivalence. It follows that the composed map
$Y_i\to \nerf_d\, \tau_d(Y_i)$ is a weak equivalence as well, thus completing
the induction step and the proof.
\end{proof}

\begin{thm}\label{thm:quillenequivnervepreopersoper}
The functor $\tau_d:\preoper\to\soper$ is a left Quillen
equivalence (where $\preoper$ is endowed with the tame model
category structure (\ref{thm:tamemodelcat})). Moreover, its right adjoint
$\nerf_d:\soper\to\preoper$ preserves and detects weak
equivalences.
\end{thm}

\begin{proof}
By virtue of Proposition \ref{prop:taudpreservestamecof},
the functor $\tau_d$ preserves cofibrations.
Moreover, the functor $\nerf_d$ preserves and detects
weak equivalences: indeed the nerve of a simplicial
operad is a Segal operad, and a morphism
of simplicial operads is fully faithful and essentially
surjective if and only if its nerve is fully faithful and
essentially surjective. One deduces from
this property of the nerve and from Proposition \ref{prop:unitpreoperwe}
that the functor $\tau_d$ sends trivial cofibrations between
cofibrant objects to trivial cofibrations.
In particular, it sends the class $\ota$
into the class of trivial cofibrations.
This implies right away that the functor $\nerf_d$
sends fibrations between fibrant objects to isofibrations
between fibrant Segal operads, whence to fibrations (Theorem \ref{thm:tamemodelcat}).
Therefore, the adjunction $(\tau_d,\nerf_d)$
is a Quillen pair; see \cite[Proposition 7.15]{JT}.
Proposition \ref{prop:unitpreoperwe} and the fact that the nerve
functor $\nerf_d$ preserves and detects weak equivalences
imply that this is in fact a Quillen equivalence.
\end{proof}

\begin{rem}
The category $\precat$ of (Segal) precategories studied by Bergner
in \cite{bergnerc} is simply $\preoper/\eta$.
Therefore, the two model structures on $\preoper$ induce
two (Quillen equivalent) model category structures on $\precat$.
Similarly, the category $\scat$ of simplicial categories
can be described as $\scat=\soper/\eta$. As both functors $\tau_d$
and $\nerf_d$ preserve the object $\eta$ (which is fibrant on both
sides of the adjunction), the preceding theorem immediately implies: 
\end{rem}

\begin{cor}[{\cite[Theorem 8.6]{bergnerc}}]
The nerve adjunction
$$\tau:\precat\rightleftarrows\scat:\nerf$$
is a Quillen equivalence (where $\precat$ is endowed with the model
category structure induced by Theorem \ref{thm:tamemodelcat}).
\end{cor}

A reformulation of the proof of Proposition \ref{prop:unitpreoperwe}
gives the following result.

\begin{thm}\label{thm:sigmacofleftproper}
Consider a pushout of simplicial operads
$$\xymatrix{
\mathcal P\ar[r]^p\ar[d]_i&\mathcal P'\ar[d]^{i'}\\
\mathcal Q\ar[r]^q&\mathcal Q'
}$$
in which the map $i$ is a cofibration, while $\mathcal P$
and $\mathcal P'$ are $\Sigma$-cofibrant.
Then this square is homotopy cocartesian.
(In particular, if the map $p$ is in addition a weak equivalence, then
$q$ is a weak equivalence as well.)
\end{thm}

\begin{proof}
Let $H$ be the class of morphisms between $\Sigma$-cofibrant simplicial operads
$i:\mathcal P\to\mathcal Q$ such that, for any
map $p:\mathcal P\to \mathcal P'$, with $\mathcal P'$
$\Sigma$-cofibrant, the pushout square
$$\xymatrix{
\mathcal P\ar[r]^p\ar[d]_i&\mathcal P'\ar[d]^{i'}\\
\mathcal Q\ar[r]^q&\mathcal Q'
}$$
is homotopy cocartesian, while $i'$ is a $\Sigma$-cofibration.
Let us first check that the class $H$ has the following properties.
\begin{itemize}
\item[(a)] For any $\Sigma$-cofibrant simplicial operad $\mathcal P$, the map
$\varnothing\to\mathcal P$ is in $H$.
\item[(b)] For any map $i:\mathcal P\to \mathcal Q$ in $H$ and any
morphism $p:\mathcal P\to \mathcal P'$ with $\mathcal P'$
$\Sigma$-cofibrant, the morphism $\mathcal P'\to \mathcal Q\amalg_{\mathcal P}\mathcal P'$
is in $H$. 
\item[(c)] For any well ordered set $I$ with smallest element $0$, and any functor $\mathcal P:I\to \soper$
such that, for any $i\in I$, the map
\smash{$\varinjlim_{j<i}\mathcal P_j\to \mathcal P_i$}
is in $H$, the map
$\mathcal P_0\to \varinjlim_{i \in I}\mathcal P_i$
is in $H$.
\item[(d)] The class $H$ is closed under retracts.
\end{itemize}
Property (a) reflects the fact that $\Sigma$-cofibrant simplicial operads are closed
under small sums as well as the stability of weak equivalences by small sums.
Property (b) comes from the fact that, for any two commutative squares of the following form in
a model category
$$\xymatrix{
X\ar[r]\ar[d]&X'\ar[r]\ar[d]&X''\ar[d]\\
Y\ar[r]&Y'\ar[r]&Y''
}$$
if the left hand square and the composed square are homotopy
cocartesian, then so is the right hand square.
Property (c) comes from the fact that homotopy cocartesian squares
are closed under homotopy colimits, because, as
weak equivalences of simplicial operads are closed under filtering colimits
(\ref{prop:wesoperfiltcolim}), any filtered colimit of simplicial operads is
weakly equivalent to the corresponding homotopy colimit.
Finally, property (d) expresses the stability of the classes of $\Sigma$-cofibrations
and of homotopy cocartesian squares under retracts.
The conclusion of this digression is that it is sufficient to prove the
theorem in the case where the map $i$ is of shape (C1) or (C2) (cf. \ref{paragr:defcofandtrivcof}):
indeed, it follows from the small object argument that any cofibration
between $\Sigma$-cofibrant simplicial operads is a retract of a transfinite
composition of pushouts of such morphisms.
As the nerve functor $\nerf_d$ is a right Quillen equivalence
which preserves weak equivalences (\ref{thm:quillenequivnervepreopersoper}),
it is sufficient to prove that the commutative square
$$\xymatrix{
\nerf_d(\mathcal P)\ar[r]^{\nerf_d(p)}\ar[d]_{\nerf_d(i)}&\nerf_d(\mathcal P')\ar[d]^{\nerf_d(i')}\\
\nerf_d(\mathcal Q)\ar[r]^{\nerf_d(q)}&\nerf_d(\mathcal Q')
}$$
is homotopy cocartesian in the model category of preoperads for which
the cofibrations are the tame cofibrations.
The last assertion of Proposition \ref{prop:computenerveenrichedtrees}
implies that $\nerf_d(i)$ is a map of shape (TC1) or (TC2) (\ref{paragr:deftamecof}),
whence, in particular, a tame cofibration. As the model category of preoperads is
left proper (\ref{thm:tamemodelcat}), we see that
Lemma \ref{lemma:pushoutnervepreoper} concludes the proof.
\end{proof}

\begin{cor}\label{cor:properoversigmacof}
For any $\Sigma$-cofibrant simplicial operad $\mathcal P$, the model
category $\soper/\mathcal P$ is proper.
\end{cor}

\begin{proof}
The property of right properness is already known (this
comes from the right properness of the model
category of simplicial operads; see \ref{thm:cmcsoper}). As for left properness,
this immediately follows from the preceding theorem and from
the fact that any simplicial operad above a $\Sigma$-cofibrant simplicial operad is
$\Sigma$-cofibrant.
\end{proof}

\begin{cor}\label{cor:propernonsymoper}
The model category of non-symmetric simplicial operads is proper.
\end{cor}

\begin{proof}
This is the preceding corollary for $\mathcal P=\mathpzc{Ass}$.
\end{proof}

\begin{cor}\label{cor:propersimpcat}
The model category of simplicial categories is proper.
\end{cor}

\begin{proof}
Apply Corollary \ref{cor:properoversigmacof} to $\mathcal P=\eta$.
\end{proof}

\begin{cor}\label{cor:quillenequivsoperproper}
Let $E_\infty$ be any $\Sigma$-cofibrant (e.g. cofibrant) simplicial
operad which is weakly equivalent to the terminal operad.
Then the model category $\soper/E_\infty$ is proper, while the
forgetful functor $\soper/E_\infty\to\soper$ is a left Quillen
equivalence.
\end{cor}

\begin{proof}
For any simplicial operad $\mathcal P$, the projection
$E_\infty\times\mathcal P\to \mathcal P$ is a weak equivalence, which
proves that the forgetful functor from $\soper/E_\infty\to\soper$ is a left Quillen
equivalence. We conclude again with Corollary \ref{cor:properoversigmacof}.
\end{proof}

\begin{rem}
We do not know if the model category of simplicial operads
is left proper or not, though.
\end{rem}

\begin{paragr}
The inclusion functor $\dset\subset\preoper$ induces an equivalence
of homotopy categories (\ref{thm:cmcreedysegaloperads})
\begin{equation}\label{eq:finaldsetpreoper}
\ho(\dset)\simeq\ho(\preoper)\, .
\end{equation}
On the other hand, we have the total right derived funtor of the homotopy coherent nerve
functor (\ref{prop:BVleftQuillen})
\begin{equation}\label{eq:finaldsetsoper}
\derR\hnerf:\ho(\soper)\to\ho(\dset)
\end{equation}
as well as the functor induced by the nerve functor (\ref{thm:quillenequivnervepreopersoper})
\begin{equation}\label{eq:finalpreopersoper}
\nerf_d:\ho(\soper)\to\ho(\preoper)
\end{equation}
(we do not need to derive the functor $\nerf_d$ because it preserves weak equivalences
on the nose).
\end{paragr}

\begin{thm}\label{thm:agreementnerves}
The composition of the nerve functor \eqref{eq:finalpreopersoper}
with the equivalence \eqref{eq:finaldsetpreoper} is canonically isomorphic to
the derived homotopy coherent nerve \eqref{eq:finaldsetsoper}.
In particular, the total right derived functor of the homotopy
coherent nerve functor is an equivalence of categories.
\end{thm}

\begin{proof}
For any fibrant simplicial operad $\mathcal P$ and any tree $T$,
as $\nerf_d$ is a right Quillen functor, we have
$$\Map(\Omega[T],\nerf_d(\mathcal P))=\Map(T,\mathcal P)$$
(because $\tau_d(\Omega[T])=T$) and, similarly, as $\hnerf$ is a right Quillen functor, we have
$$\Map(\Omega[T],\hnerf(\mathcal P))=\Map(\bv(T),\mathcal P)\, .$$
The natural weak equivalence $\bv(T)\to T$ between cofibrant objects
thus gives a functorial homotopy equivalence between Kan complexes
$$\Map(\Omega[T],\nerf_d(\mathcal P))\to\Map(\Omega[T],\hnerf(\mathcal P))\, .$$
Properties (iii) and (v) of Theorem \ref{thm:dendroidalrezkcmc}
give that $\nerf_d(\mathcal P)$ and $\hnerf(\mathcal P)$
are naturally isomorphic in $\ho(\sdset)$. Property (iv) of Theorem \ref{thm:dendroidalrezkcmc}
implies then that they are canonically isomorphic in $\ho(\preoper)$.
The last assertion follows from Theorem \ref{thm:quillenequivnervepreopersoper}.
\end{proof}

Proposition \ref{prop:BVleftQuillen} and the preceding theorem, put together, give:

\begin{thm}\label{cor:homotcohnervequillenequiv}
The adjunction
$\bv:\dset\rightleftarrows\soper:\hnerf$
is a Quillen equivalence.
\end{thm}

Both the homotopy coherent nerve functor $\hnerf$ and its left adjoint $\bv$
preserve the object $\eta$, which is fibrant (on both sides of the adjunction).
Therefore, under the identifications $\sset=\dset/\eta$
and $\scat=\soper/\eta$, we also get the following comparison result, originally
due to Joyal and Lurie:

\begin{cor}[{\cite[Theorem 2.2.5.1]{lurie}}]\label{cor:homotcohnervequillenequivsimpl}
The homotopy coherent nerve adjunction
$$\bv:\sset\rightleftarrows\scat:\mathit{hcN}$$
is a Quillen equivalence (where the category $\sset$ is endowed with
the Joyal model structure).
\end{cor}

\section{Classical operads and reduced $\infty$-operads}\label{sect:9}

\begin{paragr}
We write $\rdset$ for the full subcategory of the category $\dset$ of
dendroidal sets whose objects are the presheaves $X$ on $\Omega$
such that $X_\eta$ is the set with one element.
The objects of $\rdset$ are called \emph{reduced dendroidal sets}.

The obvious forgetful functor
\begin{equation}\label{eq:pointedforget}
\rdset\to\eta/\dset
\end{equation}
has a right adjoint
\begin{equation}
r:\eta/\dset\to\rdset
\end{equation}
defined as follows: if $X$ is a dendroidal set with a given objects $x\in X_\eta$
then $r(X)$ is the dendroidal subcomplex of $X$ whose dendrices are
obtained by the following pullbacks, for any tree $T$,
\begin{equation}
\begin{split}
\xymatrix{
r(X)_T\ar[r]\ar[d]&X_T\ar[d]\\
\ast\ar[r]^(.4){(x,\ldots,x)}&X_\eta^{E(T)}
}
\end{split}
\end{equation}
where $\ast$ denotes the set with one element, while $E(T)$ is the
set of all edges of the tree $T$, and the right vertical map is
the obvious evaluation map.

Remark that, as the functor \eqref{eq:pointedforget} is fully faithful and
commutes with colimits as well as with limits. The category $\rdset$ is also
an accessible subcategory of $\eta/\dset$.
\end{paragr}

\begin{prop}\label{prop:reductionpreservesoperads}
If $X$ is an $\infty$-operad with a distinguished object $x$, the
reduced dendroidal sets $r(X)$ is an $\infty$-operad.
\end{prop}

\begin{proof}
This immediately follows from the fact that, for any tree $T$
with a given edge $e$, the inclusion $\Lambda^e[T]\to\Omega[T]$
is bijective on objects.
\end{proof}

\begin{prop}\label{prop:reducedfibresolution}
There is functorial inner anodyne extension $X\to R(X)$
which is bijective on objects and such that $R(X)$ is an $\infty$-operad.
\end{prop}

\begin{proof}
This resolution functor is obtained by applying the small
object argument to the set of inner horn inclusions.
The fact that this does not affect objects again comes
from the fact that inner horn inclusions are all bijective on objects.
\end{proof}

\begin{prop}\label{prop:cmcreduceddset}
The category $\rdset$ of reduced dendroidal sets is endowed with a
model category structure whose weak equivalences (cofibrations)
are the morphisms which are weak equivalences (normal monomorphisms)
in the category $\dset$. The fibrant objects of this model category
are precisely the reduced dendroidal sets which are also $\infty$-operads.
\end{prop}

\begin{proof}
We will prove the existence of this model category structure
using again Jeff Smith's theorem~\cite{beke}.
Note that, as the classes of weak equivalences and of cofibrations
are detected by the forgetful functor \eqref{eq:pointedforget}, it is clear that
the class of trivial cofibrations is saturated.
The class of normal monomorphisms in $\rdset$ is generated by the
class of maps $\rho(\partial\Omega[T])\to\rho(\Omega[T])$ for any tree $T$ with
at least one vertex, where $\rho$ denotes the left adjoint of the
inclusion $\rdset\subset\dset$: for a dendroidal set $X$,
we have a pushout of the following shape.
$$\xymatrix{
\coprod_{X_\eta}\eta\ar[r]\ar[d]&X\ar[d]\\
\eta\ar[r]&\rho(X)
}$$
We see from this description that a morphism of $\rdset$
has the right lifting property with respect to normal monomorphisms in $\rdset$
if and only if it has this property in the whole category $\dset$.
In particular, such a morphism is a weak equivalence.
This achieves the proof of the existence of the model structure.

Let us prove that the fibrant objects are the reduced $\infty$-operads. It is clear that
any reduced $\infty$-operad is fibrant.
For the converse, if $X$ is a fibrant object, we have, by Proposition \ref{prop:reducedfibresolution}
an inner anodyne extension $X\to R(X)$ in $\rdset$ with $R(X)$
an $\infty$-operad. This implies that $X$ is a retract of $R(X)$ in $\rdset$,
whence is an $\infty$-operad.
\end{proof}

\begin{prop}\label{prop:embeddingclassical}
The forgetful functor $\rdset\to\eta/\dset$ induces a fully faithful functor
$\ho(\rdset)\to\ho(\eta/\dset)$.
\end{prop}

\begin{proof}
The inclusion of $\rdset$ into $\eta/\dset$
is a fully faithful left Quillen functor which preserves
weak equivalences, cofibrations, as well as fibrant objects.
This immediately implies that the induced functor
$\ho(\rdset)\to\ho(\eta/\dset)$ is fully faithful.
\end{proof}

\begin{rem}
The forgetful functor
$\ho(\rdset)\to\ho(\dset)$ is easily seen to be full and conservative;
in particular, two reduced dendroidal sets are isomorphic in $\ho(\rdset)$
if and only if they are are isomorphic in $\ho(\dset)$. 
However, this functor is not faithful. For instance, one can consider two groups
$G$ and $H$, seen as categories with one object.
Maps from $i_!\, \nerf(G)$ to $i_!\, \nerf(H)$ in $\ho(\rdset)$
are just homomorphisms of groups $G\to H$, but the set of
maps from $i_!\, \nerf(G)$ to $i_!\, \nerf(H)$ in $\ho(\dset)$
is the set of homotopy classes of maps between the
classifying spaces $BG\to BH$: we know from Hurewicz
that this is the quotient of the set of homomorphisms $G\to H$ by the
action of the group of inner automorphisms of $H$.
\end{rem}

\begin{paragr}
We write $\roper$ for the category of simplicial operads
with a single object. The objects of $\roper$ will be
called the \emph{classical operads}.
This category is endowed with a model category structure
provided by Theorem \ref{thm:cmcfixedcoloursoper}.

As both functors $\bv$ and $\hnerf$ do not affect the sets of objects, they restrict to an adjunction
\begin{equation}\label{eq:reducedhomotopycohnerve}
\bv:\rdset\rightleftarrows\roper:\hnerf\, .
\end{equation}
\end{paragr}

\begin{prop}\label{prop:reducedquillenequivdsetsoper}
The adjunction \eqref{eq:reducedhomotopycohnerve} is a Quillen
equivalence.
\end{prop}

\begin{proof}
This immediately follows from Theorem \ref{cor:homotcohnervequillenequivsimpl}.
\end{proof}

\begin{rem}
By slicing over $\eta$, the adjunction \eqref{eq:reducedhomotopycohnerve} restricts to a Quillen
equivalence between the model category of simplicial monoids and the model category
of reduced simplicial sets (whose fibrant objects are the $\infty$-categories with a single object).
\end{rem}

\begin{rem}
Similar results hold for the homotopy theory of dendroidal sets
with a fixed set of objects $C$, which is Quillen equivalent to the
category of simplicial operads with the same fixed set of objects.
There are also variations with Segal operads, and so on.
We leave these as exercises for the interested readers.
\end{rem}

\providecommand{\bysame}{\leavevmode\hbox to3em{\hrulefill}\thinspace}
\providecommand{\MR}{\relax\ifhmode\unskip\space\fi MR }
\providecommand{\MRhref}[2]{%
  \href{http://www.ams.org/mathscinet-getitem?mr=#1}{#2}
}
\providecommand{\href}[2]{#2}

\end{document}